\documentclass[a4paper]{amsart}
\pdfoutput=1
\usepackage[pdftex,
            pdfauthor={Alexis Marchand and Piotr W. Nowak},
            pdftitle={On vanishing and Hausdorffness of asymptotic cohomology},
            colorlinks=true]{hyperref}
\usepackage[T2A,T1]{fontenc}
\usepackage[english]{babel}
\usepackage[dvipsnames]{xcolor}
\usepackage{graphicx}
\usepackage{graphics}
\usepackage{enumitem}
\usepackage{amsmath,amsthm,amssymb,mathrsfs,eqnarray,relsize,bm,mathtools,dsfont}
\usepackage[msc-links]{amsrefs}
\usepackage{caption,subcaption}

\usepackage{todonotes,comment}

\usepackage{tikz}
\usetikzlibrary{cd}

\AddEnumerateCounter{\fnsymbol}{\@fnsymbol}{2}

\theoremstyle{plain}
\newtheorem{prop}{Proposition}
\numberwithin{prop}{section}

\newtheorem{lemm}[prop]{Lemma}

\newtheorem*{claim*}{Claim}

\newtheorem{innermthm}{Theorem}
\newenvironment{mthrm}[1]
  {\renewcommand\theinnermthm{#1}\innermthm}
  {\endinnermthm}

\newtheorem{innermcor}{Corollary}
\newenvironment{mcoro}[1]
  {\renewcommand\theinnermcor{#1}\innermcor}
  {\endinnermcor}

\theoremstyle{definition}
\newtheorem{defi}[prop]{Definition}

\theoremstyle{remark}
\newtheorem{remk}[prop]{Remark}

\numberwithin{equation}{section}

\renewcommand{\emptyset}{\varnothing}

\DeclareMathOperator{\Ker}{Ker}
\DeclareMathOperator{\Imm}{Im}
\DeclareMathOperator{\id}{id}
\DeclareMathOperator{\Isom}{Isom}

\DeclareMathOperator{\diff}{d}

\DeclareMathOperator{\T}{(T)}
\DeclareMathOperator{\homog}{hom}
\DeclareMathOperator{\inhom}{bar}

\DeclareMathOperator{\maxdist}{maxdist}
\DeclareMathOperator{\dHaus}{d_\mathcal{H}}

\title[On vanishing and Hausdorffness of asymptotic cohomology]{On vanishing and Hausdorffness of asymptotic cohomology}
\date{\today}
\author{Alexis Marchand}
\address{Instytut Matematyczny Polskiej Akademii Nauk, Śniadeckich 8, 00-656 Warszawa, Poland}
\email{\href{mailto:amarchand@impan.pl}{amarchand@impan.pl}}
\author{Piotr W. Nowak}
\address{Instytut Matematyczny Polskiej Akademii Nauk, Śniadeckich 8, 00-656 Warszawa, Poland}
\email{\href{mailto:pnowak@impan.pl}{pnowak@impan.pl}}

\begin{document}

\begin{abstract}
    Asymptotic cohomology was recently introduced to study stability problems in group theory.
    This paper studies vanishing and Hausdorffness questions around asymptotic cohomology.
    Our first main result is that, in a fixed degree and for a fixed asymptotic coefficient module, vanishing of asymptotic cohomology is equivalent to vanishing of reduced asymptotic cohomology.
    In the special case of diagonal coefficient modules, we also obtain a full characterisation of vanishing in terms of bounded cohomology.
    Our second main result is that asymptotic cohomology vanishes in degree $1$ for uniformly convex asymptotic Banach modules.
    From this, we can deduce Hausdorffness of degree-$2$ bounded cohomology under the same uniform convexity assumption.
\end{abstract}

\maketitle

\section{Introduction}

    Given a cohomology theory $H^\bullet$ for groups, one can often gain a lot of information about a group $\Gamma$ from the knowledge that certain cohomology groups $H^n(\Gamma;E)$ vanish.
    In this vein, the Delorme--Guichardet Theorem characterises Kazhdan's property $\T{}$ in terms of vanishing of unitary cohomology in degree $1$ for all unitary representations \cites{delorme,guichardet:dgthm,bdlhv}; similarly, Johnson's Theorem characterises amenability in terms of vanishing of bounded cohomology in degree $1$ and higher for all dual Banach representations \cites{johnson,frigerio}.

    If the cohomology groups of $H^\bullet$ have a topological structure, one can consider the \emph{reduced cohomology} $\bar{H}^\bullet$, which can be defined as the largest quotient of $H^\bullet$ whose topology is Hausdorff.
    It can then become easier to prove vanishing for $\bar{H}^\bullet$ than for $H^\bullet$.
    For instance, Shalom proved that vanishing of reduced unitary cohomology in degree $1$ for all unitary representations is enough to imply property $\T$ \cites{shalom,bdlhv}; several applications are explained in \cite{bdlhv}.
    Shalom's result was recently generalised by Bader and Sauer \cites{bader-sauer:unitary} to the context of higher-rank property $\T$.
    
    One can also ask when $H^\bullet=\bar{H}^\bullet$, which amounts to the quotient topology on $H^\bullet$ being Hausdorff; some authors say that $H^\bullet$ is \emph{reduced} when this equality holds, though we will stick to the term \emph{Hausdorff}.
    In fact, property $\T$ is also equivalent to Hausdorffness of unitary cohomology in degree $1$ for all unitary representations \cite{bdlhv}; see \cite{bader-nowak:group-alg} for higher-rank analogues of this condition.

    \medskip

    \emph{Asymptotic cohomology} was recently introduced by Glebsky, Lubotzky, Monod, and Rangarajan \cite{glmr}.
    The main motivation is that vanishing of asymptotic cohomology with certain coefficients implies stability properties for groups.
    We recall the definition of asymptotic cohomology in \S{}\ref{subsec:asymp-coh}; see also \cite{ffr} for another application to stability problems.
    The first main theorem of this paper is a strong reducedness result for asymptotic cohomology. 

	\begin{mthrm}{\ref{thrm:charact-vanish-asymp-gen}}
        \label{thrm:charact-vanish-asymp-gen-intro}
        Let $\Gamma$ be a discrete group and $\mathcal{E}\coloneqq(E_k)_{k\in\mathbb{N}}$ be an asymptotic Banach $\Gamma$-module.
   		For fixed $n\in\mathbb{N}$, the following assertions are equivalent:
		\begin{enumerate}
			\item\label{thrm:charact-vanish-asymp-gen-intro:3} $\bar{H}^n_a(\Gamma;\mathcal{E})=0$.
			\item\label{thrm:charact-vanish-asymp-gen-intro:2} $H^n_a(\Gamma;\mathcal{E})=0$.
			\item\label{thrm:charact-vanish-asymp-gen-intro:1} $H^n_a(\Gamma;\mathcal{E})=0$ and $H^{n+1}_a(\Gamma;\mathcal{E})$ is Hausdorff.
		\end{enumerate}
	\end{mthrm}

    The proof of Theorem \ref{thrm:charact-vanish-asymp-gen-intro} relies on ultraproduct techniques invented by Bader and Sauer in the context of unitary cohomology \cites{bader-sauer:below-rk,bader-sauer:unitary}.
    The general idea is that the knowledge that a codifferential map in cohomology has non-closed image can be used to construct non-vanishing cohomology classes in the ultrapower, via the Open Mapping Theorem.
    However, in our setting, we work with a general ultraproduct, not necessarily an ultrapower; hence, we need to quantify and asymptotically control the constants coming from the Open Mapping Theorem.
    This is done via an invariant that we call the \emph{closedness constant} of a bounded linear map --- see Section \ref{sec:cl-constant} for the definition and more details.

    We stress that the equivalence $\ref{thrm:charact-vanish-asymp-gen-intro:3}\Leftrightarrow\ref{thrm:charact-vanish-asymp-gen-intro:2}$ in Theorem \ref{thrm:charact-vanish-asymp-gen-intro} is surprising: it says that asymptotic cohomology vanishes if and only if \emph{reduced} asymptotic cohomology vanishes, even with a fixed coefficient module.
    The analogous statement in unitary cohomology does not hold, and work of Shalom, Bader, and Sauer \cites{shalom,bader-sauer:unitary} merely says that, in a fixed degree, vanishing of cohomology \emph{for all unitary representations} is equivalent to vanishing of reduced cohomology for all unitary representations.

    Theorem \ref{thrm:charact-vanish-asymp-gen-intro} might make it easier to prove vanishing of asymptotic cohomology, and hence stability results, in certain situations.

    \medskip

    \emph{Bounded cohomology} first emerged in work of Johnson and Trauber \cite{johnson}, and was later popularised by Gromov \cite{gromov}; it is now known to play a role in many parts of group theory and low-dimensional topology --- see for example \cite{campagnolo:hb} and references therein for introductions to various such topics.
    Refining the proof of Theorem \ref{thrm:charact-vanish-asymp-gen-intro} to the case of diagonal coefficient modules $\mathcal{E}=(E)_{k\in\mathbb{N}}$ given by the data of a single Banach $\Gamma$-module $E$, we obtain a characterisation of vanishing of diagonal asymptotic cohomology in terms of bounded cohomology.
    
    \begin{mthrm}{\ref{thrm:vanish-diag}}
        \label{thrm:vanish-diag-intro}
   		Let $\Gamma$ be a discrete group, $E$ a Banach $\Gamma$-module.
   		For $n\in\mathbb{N}$, the following assertions are equivalent:
		\begin{enumerate}
			\item\label{thrm:vanish-diag-intro:ha} $H^n_a(\Gamma;E)=0$.
			\item\label{thrm:vanish-diag-intro:hb} $H^n_b(\Gamma;E)=0$ and $H^{n+1}_b(\Gamma;E)$ is Hausdorff.
		\end{enumerate}
    \end{mthrm}

    We point out that our notations differ from those of \cite{glmr}: what we denote by $H^\bullet_a(\Gamma;E)$ is the asymptotic cohomology of $\Gamma$ with coefficients in the internal ultrapower of $E$ in the language of \cite{glmr} --- see Remark \ref{remk:compare-setups-glmr} regarding the difference between our setup and theirs.
    
    The implication $\ref{thrm:vanish-diag-intro:hb}\Rightarrow\ref{thrm:vanish-diag-intro:ha}$ of Theorem \ref{thrm:vanish-diag-intro} was proved and used extensively in \cite{glmr} (at least when $n=2$ and $\Gamma$ acts trivially on $E$, though the same proof works in full generality).
    In particular, Theorem \ref{thrm:vanish-diag-intro} says that the \emph{$2\frac{1}{2}$-condition} of \cite{glmr} is in fact equivalent to vanishing of $H^2_a(\Gamma;\mathbb{R})$.

    Theorem \ref{thrm:vanish-diag-intro} suggests that diagonal asymptotic cohomology may be a relevant tool to study Hausdorffness of bounded cohomology.

    \medskip

    Specialising to low degrees, we indeed apply Theorem \ref{thrm:vanish-diag-intro} to deduce Hausdorffness of bounded cohomology for a large class of coefficient modules.
    It is well-known that bounded cohomology vanishes in degree $1$ whenever the coefficient module is reflexive \cites{monod,johnson}.
    We show an analogous vanishing result for asymptotic cohomology in degree $1$ under a uniform convexity assumption.
    
    \begin{mthrm}{\ref{thrm:vanishing-h1a}}
        \label{thrm:vanishing-h1a-intro}
        Let $\Gamma$ be a discrete group and let $\mathcal{E}\coloneqq(E_k)_{k\in\mathbb{N}}$ be a uniformly convex asymptotic Banach $\Gamma$-module.
        Then $H_a^1(\Gamma;\mathcal{E})=0$.
    \end{mthrm}

    One proof of vanishing of bounded cohomology in degree $1$ relies on the centre of gravity, or \emph{Chebyshev centre}, of a bounded set in a uniformly convex Banach space.
    In order to adapt this proof to asymptotic cohomology, we need to show that the Chebyshev centre satisfies some continuity property, as our bounded subsets are no longer $\Gamma$-invariant, but almost invariant.
    This is the key ingredient in the proof of Theorem \ref{thrm:vanishing-h1a-intro} --- see Lemma \ref{lemm:centre}.

    It then follows from Theorems \ref{thrm:vanish-diag-intro} and \ref{thrm:vanishing-h1a-intro} that bounded cohomology is Hausdorff in degree $2$ for uniformly convex coefficient modules.
    
    \begin{mcoro}{\ref{coro:h2b-hausdorff}}
        \label{coro:h2b-hausdorff-intro}
        Let $\Gamma$ be a discrete group and let $E$ be a uniformly convex Banach $\Gamma$-module.
        Then $H_b^2(\Gamma;E)$ is Hausdorff.
    \end{mcoro}
    
    Corollary \ref{coro:h2b-hausdorff-intro} generalises a now classical theorem proved independently by Matsumoto and Morita \cite{matsumoto-morita} and Ivanov \cite{ivanov} on Hausdorffness of $H^2_b(\Gamma;\mathbb{R})$.
    Monod showed Hausdorffness of $H_b^2(\Gamma;E)$ for every separable Banach $\Gamma$-module $E$ under the condition that $\Gamma$ is finitely generated (or compactly generated locally compact second countable if one goes beyond the realm of discrete groups) \cite{monod}*{Corollary 11.4.2}; our proof is quite different from Monod's and, as far as we are aware, Hausdorffness was not previously known when $\Gamma$ is discrete but not finitely generated or when $E$ is uniformly convex but non-separable.
    
    \subsection*{Outline of the paper}
    We start in Section \ref{sec:bckgnd} by recalling the definitions of bounded cohomology, ultraproducts, and asymptotic cohomology.
    Section \ref{sec:cl-constant} is devoted to the closedness constant, which is our quantitative measure of the property of having closed image, and which we need in order to work with general ultraproducts.
    The proof of Theorem \ref{thrm:charact-vanish-asymp-gen} is completed in Section \ref{sec:vanish}; it relies on a generalised quantitative version of a lemma of Bader and Sauer which summarises their ultraproduct method; Theorem \ref{thrm:vanish-diag} on the diagonal case is deduced at the end of the section.
    Finally, Section \ref{sec:low-deg} is devoted to vanishing of asymptotic cohomology in degree $1$ (Theorem \ref{thrm:vanishing-h1a}), using a continuity result for the Chebyshev centre in uniformly convex Banach spaces; we then deduce Hausdorffness of bounded cohomology in degree $2$ (Corollary \ref{coro:h2b-hausdorff}).

    \subsection*{A remark on notations}
    Throughout this paper, the integer $n$ is always used to denote degrees in chain complexes and cohomology, while $k$ indexes sequences of coefficient modules.
    Limits along an ultrafilter should always be understood with respect to the variable $k$.

    \subsection*{Acknowledgements}
    The authors would like to thank Francesco Fournier-Facio, Clara Löh, Masato Mimura, and Roman Sauer for helpful conversations.
    
    Both authors were supported by NCN grant Maestro 13 \emph{Analysis on Groups} 2021/42/A/ST1/00306, and would also like to thank the Isaac Newton Institute for Mathematical Sciences, Cambridge, for support and hospitality during the programme \emph{Operators, Graphs, Groups}, where work on this paper was undertaken.
    This work was supported by EPSRC grant EP/Z000580/1.
    In the final phase of the preparation of this paper, the first-named author was based at Universität Regensburg with support from the Alexander von Humboldt-Stiftung; he would like to thank both for the hospitality and support.

    \subsection*{Statement on LLM usage}
    No large language model was used at any point in the preparation of this paper.

\section{Background}
    \label{sec:bckgnd}
	
	\subsection{Bounded cohomology}\label{subsec:hb}
	
	Let $\Gamma$ be a discrete group, and let $E$ be a Banach $\Gamma$-module, i.e. a Banach space together with an action of $\Gamma$ by isometries.	
	The \emph{bounded homogeneous cochain complex} of $E$ is
	\[
		0\rightarrow C^0_{b,\homog}\left(\Gamma;E\right)\xrightarrow{{\diff}^1}C^1_{b,\homog}\left(\Gamma;E\right)\xrightarrow{{\diff}^2}\cdots\xrightarrow{{\diff}^n}C^n_{b,\homog}\left(\Gamma;E\right)\xrightarrow{{\diff}^{n+1}}\cdots,
	\]
	where
	\begin{itemize}
		\item $C^n_{b,\homog}\left(\Gamma;E\right)\coloneqq\ell^\infty\left(\Gamma^{n+1},E\right)^\Gamma$, where $\ell^\infty(\Gamma^{n+1},E)$ denotes the space of uniformly bounded set-theoretic maps $\Gamma^{n+1}\rightarrow E$, equipped with the $\Gamma$-action given by
		\begin{multline}\label{eq:def-action-homog}
			(\gamma\cdot f):\left(\gamma_0,\dots,\gamma_n\right)\in\Gamma^{n+1}\longmapsto\gamma\cdot f\left(\gamma^{-1}\gamma_0,\dots,\gamma^{-1}\gamma_n\right)
			\\\textrm{for $\gamma\in\Gamma$ and $f\in\ell^\infty\left(\Gamma^{n+1},E\right)$},
		\end{multline}
		and $(-)^\Gamma$ denotes the subspace of $\Gamma$-invariant vectors, and
		\item The $n$-th codifferential ${\diff}^{n}:C^{n-1}_{b,\homog}\left(\Gamma;E\right)\rightarrow C^{n}_{b,\homog}(\Gamma;E)$ is given by
		\begin{align}\label{eq:def-codiff-homog}
			&{\diff}^{n}f:\left(\gamma_0,\dots,\gamma_{n}\right)\longmapsto\sum_{i=0}^n(-1)^if\left(\gamma_0,\dots,\hat{\gamma}_i,\dots,\gamma_{n}\right)
			&\textrm{for $f\in\ell^\infty\left(\Gamma^{n},E\right)^\Gamma$},
		\end{align}
		where $\hat{\cdot}$ denotes omission.
	\end{itemize}
    Note that each cochain group $C^n_{b,\homog}\left(\Gamma;E\right)$ has the structure of a Banach space with the $\ell^\infty$-norm, which we denote by $\|\cdot\|$.
	
	\begin{defi}\label{defi:gp-cohom}
		Let $E$ be a Banach $\Gamma$-module.
		\begin{enumerate}      
			\item The \emph{bounded cohomology} of $\Gamma$ with $E$-coefficients is defined in degree $n$ by
			\[         
				H^n_b(\Gamma;E)\coloneqq H^n\left(C^\bullet_{b,\homog}(\Gamma;E)\right)=\Ker {\diff}^{n+1}/\Imm {\diff}^n.
			\]
			\item The \emph{reduced bounded cohomology} of $\Gamma$ with $E$-coefficients is defined in degree $n$ by
			\[         
				\bar{H}_b^n(\Gamma;E)\coloneqq\bar{H}^n\left(C^\bullet_{b,\homog}(\Gamma;E)\right)=\Ker {\diff}^{n+1}/\overline{\Imm {\diff}^n}.
			\] 
		\end{enumerate}
	\end{defi}
	
	Instead of the homogeneous cochain complex, one can use an isomorphic cochain complex to compute bounded cohomology: the \emph{bounded bar cochain complex} of $E$ is
	\[   
		0\rightarrow C^0_{b,\inhom}\left(\Gamma;E\right)\xrightarrow{{\diff}^1}C^1_{b,\inhom}\left(\Gamma;E\right)\xrightarrow{{\diff}^2}\cdots\xrightarrow{{\diff}^n}C^n_{b,\inhom}\left(\Gamma;E\right)\xrightarrow{{\diff}^{n+1}}\cdots,
	\]
	where
	\begin{itemize}
		\item $C^n_{b,\inhom}\left(\Gamma;E\right)\coloneqq\ell^\infty\left(\Gamma^{n},E\right)$, and
		\item The $n$-th codifferential ${\diff}^{n}:C^{n-1}_{b,\inhom}\left(\Gamma;E\right)\rightarrow C^{n}_{b,\inhom}(\Gamma;E)$ is given by
		\begin{multline}\label{eq:def-codiff-bar}
			{\diff}^{n}f:\left(\gamma_1,\dots,\gamma_n\right)
			\longmapsto \gamma_1\cdot f\left(\gamma_2,\dots,\gamma_n\right)
			\\+\sum_{i=1}^{n-1}(-1)^if\left(\gamma_1,\dots,\gamma_i\gamma_{i+1},\dots,\gamma_{n}\right)
			+(-1)^nf\left(\gamma_1,\dots,\gamma_{n-1}\right)
			\\\textrm{for $f\in\ell^\infty\left(\Gamma^{n-1},E\right)$}.
		\end{multline}
	\end{itemize}
    Again, we denote by $\|\cdot\|$ the $\ell^\infty$ norm on $C^n_{b,\inhom}\left(\Gamma;E\right)$, which thus becomes a Banach space.

    It is a standard fact that the cochain complexes $C^\bullet_{b,\homog}(\Gamma;E)$ and $C^\bullet_{b,\inhom}(\Gamma;E)$ are isomorphic, and therefore they can both be used to compute (reduced) bounded cohomology \cite{frigerio}*{\S{}1.7}.
	
	Note that each (bounded) cohomology group is a quotient of a Banach space, and naturally inherits a topology from that Banach space.
	We will be interested in the property of this topology being Hausdorff, and we record below some elementary characterisations of this property.
	
	\begin{prop}\label{prop:charact-hausd}
		Let $\Gamma$ be a discrete group, $E$ a Banach $\Gamma$-module.
		For $n\in\mathbb{N}$, the following assertions are equivalent:
		\begin{enumerate}      
			\item $H^n_b(\Gamma;E)$ is Hausdorff.\label{prop:charact-hausd:1}
			\item The natural map $H^n_b(\Gamma;E)\rightarrow\bar{H}^n_b(\Gamma;E)$ is an isomorphism.\label{prop:charact-hausd:2}
			\item The $n$-th codifferential ${\diff}^{n}:C^{n-1}_{b,\homog}\left(\Gamma;E\right)\rightarrow C^{n}_{b,\homog}(\Gamma;E)$ has closed image.\label{prop:charact-hausd:3}
			\item The $n$-th codifferential ${\diff}^{n}:C^{n-1}_{b,\inhom}\left(\Gamma;E\right)\rightarrow C^{n}_{b,\inhom}(\Gamma;E)$ has closed image.\label{prop:charact-hausd:4}
		\end{enumerate}
	\end{prop}
	\begin{proof}   
		The equivalence of these conditions follows from the fact that a quotient of a topological vector space $V$ by a subspace $W$ is Hausdorff if and only if $W$ is closed in $V$.
	\end{proof}

    We refer the reader to Frigerio's \cite{frigerio} and Monod's \cite{monod} monographs for detailed treatments of bounded cohomology.
	
	\subsection{Ultraproducts}
	
	Recall that a \emph{filter} on the set $\mathbb{N}$ of positive integers is a non-empty set $\mathcal{F}$ of subsets of $\mathbb{N}$ --- i.e. $\mathcal{F}\subseteq 2^\mathbb{N}$ --- such that $\emptyset\not\in\mathcal{F}$, and $\mathcal{F}$ is stable under passing to supersets and under intersections.
    A subset $A\subseteq\mathbb{N}$ is called \emph{$\mathcal{F}$-large} if $A\in\mathcal{F}$.
	An \emph{ultrafilter} $\mathcal{U}$ on $\mathbb{N}$ is a filter which is maximal under inclusion; it is \emph{non-principal} if it is not of the form $\left\{A\subseteq\mathbb{N}\::\:k\in A\right\}$ for some $k\in\mathbb{N}$.
	The existence of non-principal ultrafilters on $\mathbb{N}$ is guaranteed by Zorn's Lemma.
	
	Let $\mathcal{U}$ be a non-principal ultrafilter on $\mathbb{N}$, which we fix for the rest of this paper.
	We say that a sequence $(x_k)_{k\in\mathbb{N}}$ in a topological space $X$ \emph{converges along $\mathcal{U}$} to an element $x$ in $X$ --- and we write
	\[   
		\lim_\mathcal{U} x_k=x
	\]
	--- if for every neighbourhood $W$ of $x$ in $X$, the set $\{k\in\mathbb{N}\::\:x_n\in W\}$ lies in $\mathcal{U}$.
	The key fact is that, if $X$ is compact and Hausdorff, then for every sequence $(x_k)_{k\in\mathbb{N}}$ in $X$, there exists a unique $x\in X$ such that $\lim_\mathcal{U} x_k=x$.
	
	\medskip
	
	Let $(E_k)_{k\in\mathbb{N}}$ be a sequence of Banach spaces.
	Consider the space
	\[   
		\prod_{k\in\mathbb{N}}^{\ell^\infty}E_k\coloneqq\left\{(x_k)\in\prod_{k\in\mathbb{N}}E_k\::\:\sup_{k\in\mathbb{N}}\|x_k\|<\infty\right\}.
	\]
	There is a seminorm $\|\cdot\|_\mathcal{U}$ on $\prod_{k\in\mathbb{N}}^{\ell^\infty}E_k$ given by
	\[   
		\|(x_k)\|_\mathcal{U}\coloneqq\lim_\mathcal{U}\|x_k\|.
	\]
	The subspace $N_\mathcal{U}\coloneqq\left\{(x_k)\in\prod_{k\in\mathbb{N}}^{\ell^\infty}E_k\::\:\|(x_k)\|_\mathcal{U}=0\right\}$ is closed, and we define
	\[   
		\prod_\mathcal{U} E_k\coloneqq\left(\prod_{k\in\mathbb{N}}^{\ell^\infty}E_k\right)/N_\mathcal{U}.
	\]
	The vector space $\prod_\mathcal{U} E_k$, together with the norm $\|\cdot\|_\mathcal{U}$ induced by the seminorm on $\prod_{k\in\mathbb{N}}^{\ell^\infty}E_k$, is called the \emph{(analytic) ultraproduct} of the sequence $(E_k)$.
	Given $(x_k)\in\prod_{k\in\mathbb{N}}^{\ell^\infty}E_k$, its class in $\prod_\mathcal{U} E_k$ is denoted by $[x_k]$, or by $\left[(x_k)_{k\in\mathbb{N}}\right]$.
	If all the $E_k$'s are equal to the same Banach space $E$, then $\prod_\mathcal{U} E$ is also denoted by $E^\mathcal{U}$ and called the \emph{(analytic) ultrapower} of $E$.
	We will repeatedly and implicitly used the fact that the ultraproduct $\prod_\mathcal{U} E_k$ of a sequence $(E_k)_{k\in\mathbb{N}}$ of Banach spaces is itself a Banach space.
	
	\medskip
	
	Let $\left(E_k\right)_{k\in\mathbb{N}}$ and $\left(F_k\right)_{k\in\mathbb{N}}$ be two sequences of Banach spaces.
	A sequence of bounded linear maps $\left(f_k:E_k\rightarrow F_k\right)_{k\in\mathbb{N}}$ is called \emph{uniformly bounded}
	\[   
		\sup_{k\in\mathbb{N}}\|f_k\|<\infty.
	\]
	If this is the case, then $(f_k)_{k\in\mathbb{N}}$ induces a bounded linear map $f_\mathcal{U}:\prod_\mathcal{U} E_k\rightarrow\prod_\mathcal{U} F_k$ defined by $[x_k]\mapsto\left[f_k(x_k)\right]$.

\subsection{Asymptotic cohomology}
    \label{subsec:asymp-coh}
	
	Asymptotic cohomology was introduced by Glebsky, Lubotzky, Monod, and Rangarajan \cite{glmr} as a tool to study stability questions; we recall the main definitions and point the reader to their paper \cite{glmr}*{Section 4} for a more detailed treatment.
	We focus on the case of discrete groups and work in a setup which is slightly different from that of \cite{glmr} --- see Remark \ref{remk:compare-setups-glmr} below.
	
	\medskip
	
	Given a discrete group $\Gamma$, an \emph{asymptotic Banach $\Gamma$-module} is a sequence $\mathcal{E}=\left(E_k\right)_{k\in\mathbb{N}}$ of Banach spaces, together with a sequence $\left(\pi_k:\Gamma\rightarrow\Isom(E_k)\right)_{k\in\mathbb{N}}$ of group morphisms, where $\Isom(E_k)$ is the group of linear isometries of the space $E_k$.
	Note that, for each $\gamma\in\Gamma$, the sequence $\left(\pi_k(\gamma):E_k\rightarrow E_k\right)_{k\in\mathbb{N}}$ is uniformly bounded since $\|\pi_k(\gamma)\|=1$ for each $k$.
	Therefore, there is an induced map
	\[   
		\pi_\mathcal{U}:\Gamma\rightarrow\Isom\left(\prod_\mathcal{U} E_k\right),
	\]
	which is easily verified to be a group morphism.
	This makes $\prod_\mathcal{U} E_k$ a Banach $\Gamma$-module.
	
	\begin{remk}\label{remk:compare-setups-glmr} 
		The definition of asymptotic Banach $\Gamma$-modules in \cite{glmr} might look a little different from ours at first sight, but is actually (almost) the same, and the resulting asymptotic cohomology theories will be the same.
		An \emph{internal Banach space} in their sense is an algebraic ultraproduct
		\[      
			\prod_{\mathcal{U},\mathrm{alg}}E_k\coloneqq\prod_{k\in\mathbb{N}}E_k/\left\{(x_k)\in\prod_{k\in\mathbb{N}} E_k\::\:\left\{k\in\mathbb{N}\::\:x_k=0\right\}\in\mathcal{U}\right\}
		\]
		of Banach spaces, and an \emph{internal morphism} $f_\mathcal{U}:\prod_{\mathcal{U},\mathrm{alg}}E_k\rightarrow\prod_{\mathcal{U},\mathrm{alg}}F_k$ is a map induced by a sequence $\left(f_k:E_k\rightarrow F_k\right)_{k\in\mathbb{N}}$ of bounded linear maps --- in that case, there is no uniform boundedness condition since they are working in the algebraic, rather than analytic, ultraproduct.
		The only difference with considering the category whose objects are sequences $(E_k)_{k\in\mathbb{N}}$ of Banach spaces, and whose morphisms are sequences $\left(f_k:E_k\rightarrow F_k\right)_{k\in\mathbb{N}}$ is that, in the setup of \cite{glmr}, one does not keep track of the sequence of morphisms but only remembers the induced map between the algebraic ultraproducts.
		
		In \cite{glmr}, they proceed to define an \emph{asymptotic Banach $\Gamma$-module} to be an internal isometric action $\pi:\Gamma\times\mathcal{E}\rightarrow\mathcal{E}$ which is required to induce an action on the analytic ultraproduct corresponding to the internal Banach space $\mathcal{E}$.
		As explained above, any sequence of isometric actions on Banach spaces induces an isometric action on the analytic ultraproduct because each element of $\Gamma$ has norm $1$, so uniform boundedness automatically holds.
		Hence, an asymptotic Banach module in our sense gives rise to one in their sense.
		The asymptotic cochain complex defining asymptotic cohomology (see below) then coincides with that of \cite{glmr}.
	\end{remk}
	
	Fix an asymptotic Banach $\Gamma$-module $\mathcal{E}$.
	The \emph{asymptotic homogeneous cochain complex} of $\mathcal{E}$ is
	\[   
		0\rightarrow C^0_{a,\homog}\left(\Gamma;\mathcal{E}\right)\xrightarrow{{\diff}^1_\mathcal{U}}C^1_{a,\homog}\left(\Gamma;\mathcal{E}\right)\xrightarrow{{\diff}^2_\mathcal{U}}\cdots\xrightarrow{{\diff}^n_\mathcal{U}}C^n_{a,\homog}\left(\Gamma;\mathcal{E}\right)\xrightarrow{{\diff}^{n+1}_\mathcal{U}}\cdots,
	\]
	where
	\begin{itemize}
		\item $C^n_{a,\homog}\left(\Gamma;\mathcal{E}\right)\coloneqq\left(\prod_\mathcal{U}\ell^\infty\left(\Gamma^{n+1},E_k\right)\right)^\Gamma$, where each $\ell^\infty\left(\Gamma^{n+1},E_k\right)$ is a Banach space with the $\ell^\infty$ norm $\|\cdot\|$, and the $\Gamma$-action on the ultraproduct is induced by the $\Gamma$-actions on each $\ell^\infty\left(\Gamma^{n+1},E_k\right)$ defined by $(\ref{eq:def-action-homog})$ in \S{}\ref{subsec:hb}, and
		\item The $n$-th codifferential ${\diff}^{n}_\mathcal{U}:C^{n-1}_{a,\homog}\left(\Gamma;\mathcal{E}\right)\rightarrow C^{n}_{a,\homog}(\Gamma;\mathcal{E})$ is the map induced by the sequence $\left({\diff}^n_k:\ell^\infty(\Gamma^{n},E_k)\rightarrow\ell^\infty(\Gamma^{n+1},E_k)\right)_{k\in\mathbb{N}}$ of uniformly bounded linear maps defined by $(\ref{eq:def-codiff-homog})$ in \S{}\ref{subsec:hb}.
	\end{itemize}
	
	\begin{defi}[\cite{glmr}*{Definition 4.1.9}]
		Let $\mathcal{E}$ be an asymptotic Banach $\Gamma$-module.
		\begin{enumerate}      
			\item The \emph{asymptotic cohomology} of $\Gamma$ with $\mathcal{E}$-coefficients is defined in degree $n$ by
			\[         
				H^n_a(\Gamma;\mathcal{E})\coloneqq H^n\left(C^\bullet_{a,\homog}(\Gamma;\mathcal{E})\right)=\Ker {\diff}^{n+1}_\mathcal{U}/\Imm {\diff}^n_\mathcal{U}.
			\]
			\item The \emph{reduced asymptotic cohomology} of $\Gamma$ with $\mathcal{E}$-coefficients is defined in degree $n$ by
			\[         
				\bar{H}_a^n(\Gamma;\mathcal{E})\coloneqq\bar{H}^n\left(C^\bullet_{a,\homog}(\Gamma;\mathcal{E})\right)=\Ker {\diff}^{n+1}_\mathcal{U}/\overline{\Imm {\diff}^n_\mathcal{U}}.
			\] 
		\end{enumerate}
	\end{defi}
	
	Given a Banach $\Gamma$-module $E$ (with action $\pi:\Gamma\rightarrow\Isom(E)$), we will sometimes consider the asymptotic Banach $\Gamma$-module $\mathcal{E}\coloneqq\left(E\right)_{k\in\mathbb{N}}$, with action given by the constant sequence $\left(\pi:\Gamma\rightarrow\Isom(E)\right)_{k\in\mathbb{N}}$; we will abusively denote it by $E$.
	Hence, $H^\bullet_a(\Gamma;E)$ and $\bar{H}^\bullet_a(\Gamma;E)$ should be understood as the asymptotic cohomology of $\Gamma$ with coefficients in $\mathcal{E}\coloneqq\left(E\right)_{k\in\mathbb{N}}$.
    In this situation, we say that $\mathcal{E}$ is a \emph{diagonal asymptotic Banach $\Gamma$-module} and we speak of \emph{diagonal asymptotic cohomology}.
    In the language of \cite{glmr}, this amounts to considering asymptotic cohomology with coefficients in an algebraic ultrapower on which $\Gamma$ acts diagonally.
	
	\subsection{Asymptotic cohomology via the bar cochain complex}
	\label{subsec:asymp-cohom-bar}
	
	As is the case for bounded cohomology (and group cohomology), we now explain that asymptotic cohomology can be computed using the bar cochain complex --- this is also proved in \cite{glmr}*{Remark 4.2.7}.
	
	Given an asymptotic Banach $\Gamma$-module $\mathcal{E}$, the \emph{asymptotic bar cochain complex} of $\mathcal{E}$ is
	\[   
		0\rightarrow C^0_{a,\inhom}\left(\Gamma;\mathcal{E}\right)\xrightarrow{{\diff}^1_\mathcal{U}}C^1_{a,\inhom}\left(\Gamma;\mathcal{E}\right)\xrightarrow{{\diff}^2_\mathcal{U}}\cdots\xrightarrow{{\diff}^n_\mathcal{U}}C^n_{a,\inhom}\left(\Gamma;\mathcal{E}\right)\xrightarrow{{\diff}^{n+1}_\mathcal{U}}\cdots,
	\]
	where
	\begin{itemize}
		\item $C^n_{a,\inhom}\left(\Gamma;\mathcal{E}\right)\coloneqq\prod_\mathcal{U}\ell^\infty\left(\Gamma^{n},E_k\right)$, where each $\ell^\infty\left(\Gamma^{n},E_k\right)$ is equipped with the $\ell^\infty$ norm $\|\cdot\|$, and
		\item The $n$-th codifferential ${\diff}^{n}_\mathcal{U}:C^{n-1}_{a,\inhom}\left(\Gamma;\mathcal{E}\right)\rightarrow C^{n}_{a,\inhom}(\Gamma;\mathcal{E})$ is the map induced by the sequence $\left({\diff}^n_k:\ell^\infty(\Gamma^{n-1},E_k)\rightarrow\ell^\infty(\Gamma^{n},E_k)\right)_{k\in\mathbb{N}}$ of uniformly bounded linear maps defined by $(\ref{eq:def-codiff-bar})$ in \S{}\ref{subsec:hb}.
	\end{itemize}
	
	\begin{prop}\label{prop:ha-bar}
		There is an isomorphism of cochain complexes
		\[      
			C^\bullet_{a,\homog}(\Gamma;\mathcal{E})\cong C^\bullet_{a,\inhom}(\Gamma;\mathcal{E}).
		\]
		Therefore, the (reduced) asymptotic cohomology of $\Gamma$ with $\mathcal{E}$-coefficients can be computed via
		\[      
			H^\bullet_a\left(\Gamma;\mathcal{E}\right)\cong H^\bullet\left(C^\bullet_{a,\inhom}(\Gamma;\mathcal{E})\right)
			\quad\textrm{and}\quad
			\bar{H}^\bullet_a\left(\Gamma;\mathcal{E}\right)\cong\bar{H}^\bullet\left(C^\bullet_{a,\inhom}(\Gamma;\mathcal{E})\right).
		\]
	\end{prop}
	\begin{proof}
		We define two morphisms of cochain complexes
		\[      
			u^\bullet:C^\bullet_{a,\homog}(\Gamma;\mathcal{E})\rightarrow C^\bullet_{a,\inhom}(\Gamma;\mathcal{E})
			\quad\textrm{and}\quad
			v^\bullet:C^\bullet_{a,\inhom}(\Gamma;\mathcal{E})\rightarrow C^\bullet_{a,\homog}(\Gamma;\mathcal{E})
		\]
		which will turn out to be inverse to each other.
		\begin{itemize}
			\item The $n$-th map of $u^\bullet$ is $u^n:C^n_{a,\homog}(\Gamma;\mathcal{E})\rightarrow C^n_{a,\inhom}(\Gamma;\mathcal{E})$ given by
			\begin{multline*}
				u^n\left[f_k\right]\coloneqq\left[(\gamma_1,\dots,\gamma_n)\in\Gamma^n\longmapsto f_k\left(1,\gamma_1,(\gamma_1\gamma_2),\dots,(\gamma_1\cdots\gamma_n)\right)\right]
				\\\textstyle\textrm{for $[f_k]\in\left(\prod_\mathcal{U}\ell^\infty(\Gamma^{n+1},E_k)\right)^\Gamma$}.
			\end{multline*}
			\item The $n$-th map of $v^\bullet$ is $v^n:C^n_{a,\inhom}(\Gamma;\mathcal{E})\rightarrow C^n_{a,\homog}(\Gamma;\mathcal{E})$ given by
			\begin{multline*}        
				v^n[f_k]\coloneqq\left[(\gamma_0,\dots,\gamma_{n})\in\Gamma^{n+1}\longmapsto \gamma_0\cdot f_k\left((\gamma_0^{-1}\gamma_1),\dots,(\gamma_{n-1}^{-1}\gamma_n)\right)\right]
				\\\textstyle\textrm{for $[f_k]\in\prod_\mathcal{U}\ell^\infty(\Gamma^{n},E_k)$}.
			\end{multline*}
		\end{itemize}
		An elementary computation shows that $u^\bullet$ defines a morphism of cochain complexes --- i.e. ${\diff}^{n+1}_\mathcal{U}\circ u^n=u^{n+1}\circ {\diff}^{n+1}_\mathcal{U}$ for all $n$ --- and the same is true of $v^\bullet$.
		One readily checks that $u^\bullet\circ v^\bullet=\id_{C^\bullet_{a,\inhom}(\Gamma;\mathcal{E})}$.
		Moreover, $v^\bullet\circ u^\bullet$ is the endomorphism of $C^\bullet_{a,\homog}(\Gamma;\mathcal{E})$ given by
		\begin{multline*} 
			v^nu^n[f_k]=\left[(\gamma_0,\dots,\gamma_n)\mapsto\gamma_0\cdot f_k\left(1,(\gamma_0^{-1}\gamma_1),\dots,(\gamma_0^{-1}\gamma_{n})\right)\right]
			=[f_k].
            \\\textstyle\textrm{for $[f_k]\in\left(\prod_\mathcal{U}\ell^\infty(\Gamma^{n+1};E)\right)^\Gamma$.}
		\end{multline*}
		Therefore $v^\bullet\circ u^\bullet=\id_{C^\bullet_{a,\homog}(\Gamma;\mathcal{E})}$, and $u^\bullet$, $v^\bullet$ are mutually inverse isomorphisms as wanted.
	\end{proof}

    Note that the isomorphisms of cochain complexes in the proof of Proposition \ref{prop:ha-bar} are simply the maps between ultraproducts induced by the usual isomorphisms between the homogeneous and bar cochain complexes in bounded cohomology \cite{frigerio}*{\S{}1.7}.

    Throughout the rest of this paper, we will always work with the bar cochain complexes for bounded and asymptotic cohomology.
    We will often use the fact that $C^\bullet_{a,\inhom}(\Gamma;E)$ is obtained from $C^\bullet_{b,\inhom}(\Gamma;E)$ by taking ultrapowers: $C^n_{a,\inhom}(\Gamma,{E})=C^n_{b,\inhom}(\Gamma;E)^\mathcal{U}$ for each $n$, and ${\diff}^n_\mathcal{U}:C^{n-1}_{b,\inhom}(\Gamma;E)^\mathcal{U}\rightarrow C^{n}_{b,\inhom}(\Gamma;E)^\mathcal{U}$ is the map induced by ${\diff}^n:C^{n-1}_{b,\inhom}(\Gamma;E)\rightarrow C^{n}_{b,\inhom}(\Gamma;E)$ between ultrapowers.

\section{Quantifying closedness}
    \label{sec:cl-constant}

    Given a bounded linear map $f:E\rightarrow F$ between Banach spaces, we introduce an invariant --- the closedness constant --- associated to $f$ that measures the property of $f$ having closed image.
    This invariant is defined in \S{}\ref{subsec:cl-op}; then we prove a few lemmas for later use in \S{}\ref{subsec:cl-op-ultraprod}, \S{}\ref{subsec:exactness-ultraprod}, and \S{}\ref{subsec:cl-dual}.
    In \S{}\ref{subsec:asymp-unif-cond}, we use the closedness constant to define conditions which are asymptotic variations on Matsumoto and Morita's uniform boundary condition \cite{matsumoto-morita}, and which we will relate to asymptotic cohomology.
    
    \medskip
    
    We learnt after finishing this paper that Elena Bogliolo had introduced an invariant, which she calls the \emph{asymptotic vanishing modulus} \cite{bogliolo}, and which is closely related to the what is for us the ultralimit of the closedness constant of the codifferentials in bounded cohomology, appearing in our Definition \ref{defi:asymp-hausd-vanish}.

    \subsection{Closedness constant}
    \label{subsec:cl-op}
    We first define our invariant.
    
    \begin{defi}
        \label{defi:cl}
        Given a non-zero map $f:E\rightarrow F$ between Banach spaces, the \emph{closedness constant} of $f$ is
        \[
            \mathscr{C}(f)\coloneqq\sup_{x\in E\smallsetminus\Ker f}\left(\inf_{z\in\Ker f}\frac{\|x-z\|}{\|f(x)\|}\right)\in[0,+\infty].
        \]
        By convention, the closedness constant of the zero map will be zero.
    \end{defi}

    The following proposition justifies the name; it essentially follows from the Open Mapping Theorem, see for instance \cite{abramovich-aliprantis}*{Corollary 2.15}.

    \begin{prop}
        \label{prop:cl-constant-name}
        Given a bounded linear map $f:E\rightarrow F$ between Banach spaces, the following assertions are equivalent:
        \begin{enumerate}
            \item $\mathscr{C}(f)<+\infty$.
            \item $f$ has closed image.
        \end{enumerate}
    \end{prop}

    \subsection{Closedness constant and ultraproducts}
    \label{subsec:cl-op-ultraprod}

    The next lemma gives an upper bound on the closedness constant of a map between ultraproducts induced by a uniformly bounded sequence of linear maps.
    This will be useful to show Hausdorffness of asymptotic cohomology in certain situations.
    
    \begin{lemm}[Closedness constant of an ultraproduct map]
		\label{lemm:closed-img-ultrap-gen}
		Let $(f_k:E_k\rightarrow F_k)_{k\in\mathbb{N}}$ be a uniformly bounded sequence of linear maps between Banach spaces, and let $f_\mathcal{U}:\prod_\mathcal{U}E_k\rightarrow \prod_\mathcal{U}F_k$ be the induced map between ultraproducts.
        Then
        \[
            \mathscr{C}(f_\mathcal{U})\leq\lim_\mathcal{U}\mathscr{C}(f_k).
        \]
	\end{lemm}
    \begin{proof}
        Let $\varepsilon>0$, and let $x=[x_k]\in\prod_\mathcal{U}E_k\smallsetminus\Ker f_\mathcal{U}$.
        For $k\in\mathbb{N}$, by definition of $\mathscr{C}$, there exists $z_k\in\Ker f_k$ such that
        \[
            \|x_k-z_k\|\leq\left(\mathscr{C}(f_k)+\varepsilon\right)\left\|f_k(x_k)\right\|.
        \]
        The sequence $(z_k)_{k\in\mathbb{N}}$ is bounded, and we consider $z\coloneqq[z_k]\in\prod_\mathcal{U}E_k$.
        We have
        \begin{align*}
            \|x-z\|_\mathcal{U}
            &=\lim_\mathcal{U}\|x_k-z_k\|
            \\&\leq\lim_\mathcal{U}\left(\left(\mathscr{C}(f_k)+\varepsilon\right)\left\|f_k(x_k)\right\|\right)
            \\&=\left(\lim_\mathcal{U}\mathscr{C}(f_k)+\varepsilon\right)\left\|f_\mathcal{U}(x)\right\|.
        \end{align*}
        Moreover, it is clear that $z\in\Ker f_\mathcal{U}$.
        Looking back at the definition of $\mathscr{C}$, we have thus shown that
        \[
            \mathscr{C}(f_\mathcal{U})\leq\lim_\mathcal{U}\mathscr{C}(f_k)+\varepsilon.
        \]
        Since this holds for all $\varepsilon>0$, it follows that $\mathscr{C}(f_\mathcal{U})\leq\lim_\mathcal{U}\mathscr{C}(f_k)$.
    \end{proof}

    \subsection{Closedness constant and duals}
    \label{subsec:cl-dual}

    It is a classical theorem of Banach \cite{banach}*{pp. 149-150} (see also \cite{abramovich-aliprantis}*{Theorem 2.18}) that a bounded linear map $f:E\rightarrow F$ between Banach spaces has closed image if and only if the dual map $f^*:F^*\rightarrow E^*$ has closed image.
    In the language introduced above, this means that $\mathscr{C}(f)<+\infty$ if and only if $\mathscr{C}(f^*)<+\infty$.
    The following strengthens one implication of this theorem.

    \begin{lemm}[Closedness constant of a dual map]
        \label{lemm:cl-dual}
        Let $f:E\rightarrow F$ be a bounded linear map between Banach spaces.
        Assume that $f$ is injective with dense image.
        Then
        \[
            \mathscr{C}(f^*)\geq\mathscr{C}(f).
        \]
    \end{lemm}
    \begin{proof}
        Since $f$ is injective with dense image, so is $f^*$.
        In particular, the closedness constant can be expressed as
        \[
            \mathscr{C}(f)=\sup_{x\in E\smallsetminus\Ker f}\frac{\|x\|}{\left\|f(x)\right\|},
        \]
        and similarly for $f^*$.

        Let $x\in E\smallsetminus\Ker f$.
        Let $\varepsilon>0$.
        The Hahn--Banach Theorem yields an element $\phi\in E^*$ such that $|\phi(x)|=\|\phi\|\cdot\|x\|$; we can assume after renormalising that $\|\phi\|=1$.
        Since $f^*$ has dense image, we can find $\psi\in F^*$ such that $\|\phi-f^*\psi\|\leq\varepsilon=\varepsilon\|\phi\|$.
        This implies that
        \[
            \|\phi-f^*\psi\|\leq\varepsilon\|\phi\|\leq\varepsilon\|f^*\psi\|+\varepsilon\|\phi-f^*\psi\|,
        \]
        so $\|\phi-f^*\psi\|\leq\frac{\varepsilon}{1-\varepsilon}\|f^*\psi\|$.
        Hence,
        \begin{align*}
            \|\psi\|\cdot\|f(x)\|
            &\geq|\psi\circ f(x)|
            \\&\geq|\phi(x)|-\|\phi-f^*\psi\|\cdot\|x\|
            \\&=\left(\|\phi\|-\|\phi-f^*\psi\|\right)\cdot\|x\|
            \\&\geq\left(\|f^*\psi\|-2\|\phi-f^*\psi\|\right)\cdot\|x\|
            \\&\geq\left(1-\frac{2\varepsilon}{1-\varepsilon}\right)\cdot\|f^*\psi\|\cdot\|x\|.
        \end{align*}
        It follows that
        \[
            \frac{\|x\|}{\|f(x)\|}
            \leq\left(1-\frac{2\varepsilon}{1-\varepsilon}\right)^{-1}\frac{\|\psi\|}{\|f^*\psi\|}
            \leq\left(1-\frac{2\varepsilon}{1-\varepsilon}\right)^{-1}\mathscr{C}(f^*).
        \]
        Taking the supremum over $x$ yields
        \[
            \mathscr{C}(f)\leq\left(1-\frac{2\varepsilon}{1-\varepsilon}\right)^{-1}\mathscr{C}(f^*).
        \]
        Since this holds for all $\varepsilon>0$, we deduce that $\mathscr{C}(f)\leq\mathscr{C}(f^*)$.
    \end{proof}

    \begin{remk}
        One could remove from Lemma \ref{lemm:cl-dual} the assumption that $f$ is injective with dense image.
        Moreover, the reverse inequality should also hold.
        However, we will not need any of these facts, so we leave their proofs to the reader.
    \end{remk}

    \subsection{Exactness of ultraproducts}
    \label{subsec:exactness-ultraprod}

    In the next lemma, uniform control of the closedness constant allows us to obtain an exactness result for ultraproducts, which we will use at several points.

    \begin{lemm}[Exactness of ultraproducts]
        \label{lemm:ultraprod-exact}
        For each $k\in\mathbb{N}$, let
        \[
            A_k\xrightarrow{f_k}B_k\xrightarrow{g_k}C_k
        \]
        be a sequence of Banach spaces, in the sense that $\Imm f_k=\Ker g_k$.
        Assume that this sequence is exact for every $k$ in a $\mathcal{U}$-large set.
        Assume in addition that the sequences $\left(f_k\right)_{k\in\mathbb{N}}$, $\left(g_k\right)_{k\in\mathbb{N}}$ are uniformly bounded, and that
        \[
            \lim_\mathcal{U}\mathscr{C}(f_k),\lim_\mathcal{U}\mathscr{C}(g_k)<+\infty.
        \]
        Then the induced sequence of Banach spaces
        \[
            \prod_\mathcal{U}A_k\xrightarrow{f_\mathcal{U}}\prod_\mathcal{U}B_k\xrightarrow{g_\mathcal{U}}\prod_\mathcal{U}C_k
        \]
        is exact.
    \end{lemm}
    
    \begin{proof}
        Since $g_k\circ f_k=0$ for all $k$ in a $\mathcal{U}$-large set, it is clear that $g_\mathcal{U}\circ f_\mathcal{U}=0$, i.e. $\Imm f_\mathcal{U}\subseteq\Ker g_\mathcal{U}$; it remains to show the reverse inclusion.
        Let $b\coloneqq[b_k]\in\Ker g_\mathcal{U}\subseteq\prod_\mathcal{U}B_k$, where the sequence $(b_k)_{k\in\mathbb{N}}$ is bounded.
        For each $k$ in a $\mathcal{U}$-large set, the definition of $\mathscr{C}$ (see Definition \ref{defi:cl}) yields an element $z_k\in\Ker g_k$ such that
        \[
            \|b_k-z_k\|\leq\left(\mathscr{C}(g_k)+1\right)\left\|g_k(b_k)\right\|.
        \]
        By exactness of $A_k\xrightarrow{f_k}B_k\xrightarrow{g_k}C_k$ for $k$ in a $\mathcal{U}$-large set, and by definition of $\mathscr{C}$, there is an element $a_k\in A_k$ such that
        \[
            z_k=f_k(a_k)
            \quad\textrm{and}\quad
            \|a_k\|\leq\left(\mathscr{C}(f_k)+1\right)\|z_k\|.
        \]
        For any index $k$ not in one of the $\mathcal{U}$-large sets chosen above, we simply choose $a_k\coloneqq0$.
        Since $\lim_\mathcal{U}\mathscr{C}(f_k),\lim_\mathcal{U}\mathscr{C}(g_k)<+\infty$, it follows that the sequence $(a_k)_{k\in\mathbb{N}}$ is bounded, so it defines an element $a=[a_k]\in\prod_\mathcal{U}A_k$.
        Now we have
        \[
            \left\|b-f_\mathcal{U}(a)\right\|_\mathcal{U}=\lim_{\mathcal{U}}\left\|b_k-f_k(a_k)\right\|\leq\lim_\mathcal{U}\left(\mathscr{C}(g_k)+1\right)\left\|g_k(b_k)\right\|=0
        \]
        since $\lim_\mathcal{U}\|g_k(b_k)\|=\|g_\mathcal{U}(b)\|_\mathcal{U}=0$.
        This proves that $b\in\Imm f_\mathcal{U}$ as wanted.
    \end{proof}

    \begin{remk}
        \begin{enumerate}
            \item Lemma \ref{lemm:ultraprod-exact} generalises \cite{bader-sauer:below-rk}*{Lemma 30}, which is the special case of ultraproducts.
            In the statement of \cite{bader-sauer:below-rk}*{Lemma 30}, the exact sequence automatically gives that the maps $f,g$ have closed image, or in other words that $\mathscr{C}(f),\mathscr{C}(g)<+\infty$ (see Proposition \ref{prop:cl-constant-name}).
            \item Boundedness of $\left(f_k\right)_{k\in\mathbb{N}}$ and $\left(g_k\right)_{k\in\mathbb{N}}$ in Lemma \ref{lemm:ultraprod-exact} are merely needeed to guarantee the existence of induced maps between ultraproducts.
        \end{enumerate}
    \end{remk}

    \subsection{Asymptotic vanishing and Hausdorffness}
    \label{subsec:asymp-unif-cond}

    Given a fixed degree $n\in\mathbb{N}$, a discrete group $\Gamma$ and an asymptotic Banach $\Gamma$-module $\mathcal{E}=(E_k)_{k\in\mathbb{N}}$, we introduce two conditions that capture the properties of the bounded cohomology groups $(H^n_b(\Gamma;E_k))_{k\in\mathbb{N}}$ vanishing, or being Hausdorff, in an asymptotically uniform manner.
    These conditions are related to Matsumoto and Morita's uniform boundary condition \cite{matsumoto-morita} (see also \cites{monod-nariman,fflm,loeh:anote}); they are based on the closedness constant, and we will relate them to asymptotic cohomology.

    \begin{defi}
    	\label{defi:asymp-hausd-vanish}
        Let $\Gamma$ be a discrete group and $\mathcal{E}\coloneqq(E_k)_{k\in\mathbb{N}}$ be an asymptotic Banach $\Gamma$-module.
        Fix $n\in\mathbb{N}$.
        \begin{itemize}
            \item $\mathcal{E}$ has \emph{asymptotically Hausdorff $n$-bounded cohomology} if
            \[
                \lim_\mathcal{U}\mathscr{C}\left({\diff}^n_k:C_{b,\inhom}^{n-1}(\Gamma;E_k)\rightarrow C_{b,\inhom}^{n}(\Gamma;E_k)\right)<+\infty.
            \]
            \item $\mathcal{E}$ has \emph{asymptotically vanishing $n$-bounded cohomology} if
            \begin{itemize}[label=$\circ$]
                \item The set $\{k\in\mathbb{N}\::\:H^n_b(\Gamma;E_k)=0\}$ lies in $\mathcal{U}$, and
                \item $\lim_\mathcal{U}\mathscr{C}\left({\diff}^n_k:C_{b,\inhom}^{n-1}(\Gamma;E_k)\rightarrow C_{b,\inhom}^{n}(\Gamma;E_k)\right)<+\infty$.
            \end{itemize}
        \end{itemize}
    \end{defi}

    In the case of diagonal asymptotic modules, the following proposition says that these conditions correspond to usual Hausdorffness and vanishing conditions.

    \begin{prop}
        \label{prop:charact-unif-cond-diag}
        Let $\Gamma$ be a discrete group, $E$ a Banach $\Gamma$-module, and $n\in\mathbb{N}$.
        \begin{enumerate}
            \item\label{prop:charact-unif-cond-diag:hausd} $H_b^n(\Gamma;E)$ is Hausdorff if and only if $\mathcal{E}=(E)_{k\in\mathbb{N}}$ has asymptotically Hausdorff $n$-bounded cohomology.
            \item\label{prop:charact-unif-cond-diag:vanish} $H_b^n(\Gamma;E)=0$ if and only if $\mathcal{E}=(E)_{k\in\mathbb{N}}$ has asymptotically vanishing $n$-bounded cohomology.
        \end{enumerate}
    \end{prop}
    \begin{proof}
        \begin{enumerate}
            \item[\ref{prop:charact-unif-cond-diag:hausd}] The condition that $\mathscr{C}({\diff}^n)<+\infty$ is equivalent to ${\diff}^n$ having closed image (see Proposition \ref{prop:cl-constant-name}), which is equivalent to $H^n_b(\Gamma;E)$ being Hausdorffness (see Proposition \ref{prop:charact-hausd}).
            \item[\ref{prop:charact-unif-cond-diag:vanish}] The implication $(\Leftarrow)$ is clear; $(\Rightarrow)$ says that $\mathscr{C}({\diff}^n)<+\infty$ whenever $H_b^n(\Gamma;E)=0$, which is clear from item \ref{prop:charact-unif-cond-diag:hausd} above.\qedhere
        \end{enumerate}
    \end{proof}

    The next lemma gives a first relation between the above conditions and asymptotic cohomology.
    
    \begin{lemm}[Hausdorffness of $H^n_a$ from asymptotic Hausdorffness]
		\label{lemm:hausdorff-ha-hb-gen}
		Let $\Gamma$ be a discrete group, $\mathcal{E}\coloneqq(E_k)_{k\in\mathbb{N}}$ an asymptotic Banach $\Gamma$-module, and $n\in\mathbb{N}$.
		If $\mathcal{E}$ has asymptotically Hausdorff $n$-bounded cohomology, then $H_a^n(\Gamma;\mathcal{E})$ is Hausdorff.
	\end{lemm}
	\begin{proof}
		Lemma \ref{lemm:closed-img-ultrap-gen} yields
        \[
            \mathscr{C}({\diff^n_\mathcal{U}})
            \leq\lim_\mathcal{U}\mathscr{C}\left(\diff^n_k:C_{b,\inhom}^{n-1}(\Gamma;E_k)\rightarrow C_{b,\inhom}^{n}(\Gamma;E_k)\right)
            <+\infty,
        \]
        using the asymptotically uniform $n$-cocycles condition.
        Hence, ${\diff^n_\mathcal{U}}$ has closed image, which means that $H_a^n(\Gamma;\mathcal{E})$ is Hausdorff similarly to Proposition \ref{prop:charact-hausd}.
	\end{proof}

\section{Vanishing of (reduced) asymptotic cohomology}
    \label{sec:vanish}

    This section is devoted to the proof of Theorem \ref{thrm:charact-vanish-asymp-gen}, which asserts that vanishing of asymptotic cohomology in a fixed degree is equivalent to vanishing of reduced asymptotic cohomology in the same degree, and also automatically implies Hausdorffness in the next degree.
    This is based on a generalisation of a lemma of Bader and Sauer \cite{bader-sauer:below-rk}*{Lemma 32}, which is the content of \S{}\ref{subsec:bader-sauer}.
    Theorem \ref{thrm:charact-vanish-asymp-gen} is then proved in \S{}\ref{subsec:thrmA}.
    In \S{}\ref{subsec:nhalf}, we generalise a criterion for vanishing of diagonal asymptotic cohomology of Glebsky, Lubotzky, Monod, and Rangarajan \cite{glmr}*{Proposition 4.2.9}; this will allow us to specialise Theorem \ref{thrm:charact-vanish-asymp-gen} to the case of diagonal coefficients and obtain Theorem \ref{thrm:vanish-diag}, a full characterisation of vanishing of diagonal asymptotic cohomology in terms of bounded cohomology, in \S{}\ref{subsec:diag}.

    \subsection{Quantitative Bader--Sauer Lemma}
    \label{subsec:bader-sauer}
    In \cite{bader-sauer:unitary}*{Theorem 3.7}, Bader and Sauer proposed an ultraproduct argument to prove and generalise a theorem of Shalom \cite{shalom} characterising property $\T{}$ in terms of vanishing of \emph{reduced} cohomology for all unitary representations.
    In a subsequent paper, they summarised the core of the argument in one lemma \cite{bader-sauer:below-rk}*{Lemma 32}.
    The following is a generalisation of their lemma to ultraproducts (rather than ultrapowers); the condition on the map having non-closed image must now be replaced with a condition on the closedness constant that we introduced in Section \ref{sec:cl-constant}.
    
    \begin{lemm}[Quantitative Bader--Sauer]
        \label{lemm:bader-sauer-gen}
        Let $(f_k:E_k\rightarrow F_k)_{k\in\mathbb{N}}$ be a uniformly bounded sequence of linear maps between Banach spaces, and let $f_\mathcal{U}:\prod_\mathcal{U}E_k\rightarrow \prod_\mathcal{U}F_k$ be the induced map between ultraproducts.
		Suppose that
        \[
            \lim_\mathcal{U}\mathscr{C}(f_k)=+\infty.
        \]
        Then
		\begin{enumerate}
			\item\label{lemm:bader-sauer-gen:ker} $\prod_\mathcal{U}\Ker f_k\subsetneq\Ker f_\mathcal{U}$ --- in particular, $f_\mathcal{U}$ is not injective --- and
			\item\label{lemm:bader-sauer-gen:im} $\overline{\Imm f_\mathcal{U}}\subsetneq\prod_\mathcal{U}\overline{\Imm f_k}$ --- in particular, $f_\mathcal{U}$ has non-dense image.
		\end{enumerate}
    \end{lemm}
    \begin{proof}
        We first assume that each map $f_k$ is injective with dense image; exactness of ultraproducts (Lemma \ref{lemm:ultraprod-exact}) will then allow us to remove these assumptions.
        For each $k\in\mathbb{N}$, the definition of $\mathscr{C}$ (see Definition \ref{defi:cl}) yields an element $x_k\in E_k$ such that
        \[
            \|x_k\|\geq\min\left\{\left(\mathscr{C}(f_k)-1\right),k\right\}\left\|f_k(x_k)\right\|.
        \]
        (The minimum in the above inequality is to cover the case where $\mathscr{C}(f_k)=+\infty$.)
        We can assume after renormalising that $\|x_k\|=1$.
        Hence, $x=[x_k]$ is a unit vector of the ultraproduct $\prod_\mathcal{U}E_k$, and we have
        \[
            \left\|f_\mathcal{U}(x)\right\|_\mathcal{U}=\lim_\mathcal{U}\left\|f_k(x_k)\right\|\leq\lim_\mathcal{U}\left(\min\left\{\left(\mathscr{C}(f_k)-1\right),k\right\}\right)^{-1}=0,
        \]
        so $x\in\Ker f_\mathcal{U}\smallsetminus\{0\}$, which proves \ref{lemm:bader-sauer-gen:ker} in the case where each $f_k$ is injective (with dense image).

        For \ref{lemm:bader-sauer-gen:im} (still assuming that each $f_k$ is injective with dense image), we consider the dual map $f_k^*:F_k^*\rightarrow E_k^*$ of each $f_k$; it is also injective with dense image, and Lemma \ref{lemm:cl-dual} guarantees that $\lim_\mathcal{U}\mathscr{C}(f_k^*)=+\infty$.
        Thus, the above argument yields a unit vector $\psi$ in the kernel of $(f^*)_\mathcal{U}:\prod_\mathcal{U}F_k^*\rightarrow\prod_\mathcal{U}E_k^*$.
        The image of $\psi$ under the natural isometric embedding $\prod_\mathcal{U}F_k^*\rightarrow(\prod_\mathcal{U}F_k)^*$ also has norm $1$ but vanishes on $\Imm f_\mathcal{U}$, hence also on $\overline{\Imm f_\mathcal{U}}$.
        It follows that $\overline{\Imm f_\mathcal{U}}\subsetneq\prod_\mathcal{U}F_k=\prod_\mathcal{U}\overline{\Imm f_k}$, completing the proof of \ref{lemm:bader-sauer-gen:im} when each $f_k$ is injective with dense image.
        
        For the general case, consider for each $k\in\mathbb{N}$ the quotient $\bar{E}_k\coloneqq E_k/\Ker f_k$, and let $\bar{f}_k:\bar{E}_k\rightarrow\overline{\Imm f_k}$ be the map induced by $f_k$.
        Note that each map $\bar{f}_k$ is injective with dense image, and $\lim_\mathcal{U}\mathscr{C}(\bar{f}_k)=\lim_\mathcal{U}\mathscr{C}(f_k)=+\infty$, so the above proof shows that the induced map $\bar{f}_\mathcal{U}:\prod_\mathcal{U}\bar{E}_k\rightarrow\prod_\mathcal{U}\overline{\Imm f_k}$ satisfies strict inclusions
        \begin{equation}
            \label{eq:bader-sauer-sp-case}
            \Ker\bar{f}_\mathcal{U}\supsetneq\{0\}
            \quad\textrm{and}\quad
            \overline{\Imm\bar{f}_\mathcal{U}}\subsetneq\prod_\mathcal{U}\overline{\Imm f_k}.
        \end{equation}
        For each $k\in\mathbb{N}$, there is a short exact sequence of Banach spaces
        \[
            0\rightarrow\Ker f_k\xrightarrow{\iota_k}E_k\xrightarrow{\pi_k}\bar{E}_k\rightarrow0,
        \]
        where $\iota_k$ is the inclusion and $\pi_k$ is the projection.
        Since $\Ker f_k$ and $\bar{E}_k$ inherit their respective norms from $E_k$, it is easy to verify that the norm and closedness constant of $\iota_k$ and $\pi_k$ are at most $1$.
        Hence, Lemma \ref{lemm:ultraprod-exact} applies and yields a short exact sequence of ultraproducts.
        Using the obvious inclusion $\prod_\mathcal{U}\Ker f_k\subseteq\Ker f_\mathcal{U}$, we get a commutative diagram of Banach spaces with exact rows and columns:
        \[\begin{tikzcd}[row sep=normal]
            && 0 \arrow[d] & 0 \arrow[d]\\
            0 \arrow[r] & \prod_\mathcal{U}\Ker f_k \arrow[r]\arrow[d,equal] & \Ker f_\mathcal{U} \arrow[r]\arrow[d] & \Ker\bar{f}_\mathcal{U} \arrow[r]\arrow[d] & 0\\
            0\arrow[r] & \prod_\mathcal{U}\Ker f_k \arrow[r,"\iota_\mathcal{U}"] & \prod_\mathcal{U}E_k \arrow[r,"\pi_\mathcal{U}"]\arrow[d,"f_\mathcal{U}"] & \prod_\mathcal{U}\bar{E}_k \arrow[r]\arrow[d,"\bar{f}_\mathcal{U}"] & 0\\
            && \prod_\mathcal{U}\overline{\Imm f_k} \arrow[r,equal] & \prod_\mathcal{U}\overline{\Imm f_k}.
        \end{tikzcd}\]
        But we know from $(\ref{eq:bader-sauer-sp-case})$ that $\Ker\bar{f}_\mathcal{U}\neq\{0\}$, which implies $\prod_\mathcal{U}\Ker f_k\subsetneq\Ker f_\mathcal{U}$ by exactness of the second row, proving \ref{lemm:bader-sauer-gen:ker} in the general case.
        As for \ref{lemm:bader-sauer-gen:im}, note that $\overline{\Imm f_\mathcal{U}}=\overline{\Imm\bar{f}_\mathcal{U}}\subsetneq\prod_\mathcal{U}\overline{\Imm f_k}$.
    \end{proof}

    \begin{remk}
        If $f:E\rightarrow F$ is a bounded linear map between Banach spaces, then $\mathscr{C}(f)=+\infty$ if and only if $f$ has non-closed image.
        Hence, applying Lemma \ref{lemm:bader-sauer-gen} to the constant sequence $(f:E\rightarrow F)_{k\in\mathbb{N}}$ recovers \cite{bader-sauer:below-rk}*{Lemma 32}.
    \end{remk}

    \subsection{Proof of Theorem \ref{thrm:charact-vanish-asymp-gen}}
    \label{subsec:thrmA}

    The first step towards the proof of Theorem \ref{thrm:charact-vanish-asymp-gen} is the following lemma, which is an application of our quantitative version of the Bader--Sauer Lemma (\ref{lemm:bader-sauer-gen}).

    \begin{lemm}
        \label{lemm:vanish-hausdorff-gen}
		Let $\Gamma$ be a discrete group, $\mathcal{E}\coloneqq(E_k)_{k\in\mathbb{N}}$ an asymptotic Banach $\Gamma$-module, and $n\in\mathbb{N}$.
		If the reduced diagonal asymptotic cohomology $\bar{H}_a^n(\Gamma;E)$ vanishes, then
        \begin{enumerate}
            \item $\mathcal{E}$ has asymptotically Hausdorff $n$-bounded cohomology, and
            \item $\mathcal{E}$ has asymptotically Hausdorff $(n+1)$-bounded cohomology.
        \end{enumerate}
	\end{lemm}
	
	\begin{proof}
		We prove the contrapositive: if $\lim_\mathcal{U}\mathscr{C}({\diff^n_k})=+\infty$ or $\lim_\mathcal{U}\mathscr{C}({\diff^{n+1}_k})=+\infty$, then $\bar{H}_a^n\left(\Gamma;{E}\right)\neq0$.
		
		Recall from \S{}\ref{subsec:asymp-cohom-bar} that the asymptotic bar cochain complex $C^\bullet_{a,\inhom}(\Gamma;\mathcal{E})$ is obtained from $(C^\bullet_{b,\inhom}(\Gamma;E_k))_{k\in\mathbb{N}}$ by taking the ultraproduct, so that we have
		\[      
			\bar{H}_a^n\left(\Gamma;{E}\right)=\Ker{\diff}^{n+1}_\mathcal{U}/\overline{\Imm{\diff}^{n}_\mathcal{U}},
		\]
        where ${\diff^{n+1}_\mathcal{U}}:\prod_\mathcal{U}C^n_{b,\inhom}(\Gamma;E_k)\rightarrow\prod_\mathcal{U}C^{n+1}_{b,\inhom}(\Gamma;E_k)$ is the map between ultraproducts induced by $({\diff}^{n+1}_k:C^n_{b,\inhom}(\Gamma;E_k)\rightarrow C^{n+1}_{b,\inhom}(\Gamma;E_k))_{k\in\mathbb{N}}$, and similarly for ${\diff^n_\mathcal{U}}$.
		
		Now write the chain of inclusions
		\begin{equation}
			\label{eq:chain-incl-gen}
			\overline{\Imm{\diff}^{n}_\mathcal{U}}
			\subseteq\prod_\mathcal{U}\overline{\Imm{\diff}^{n}_k}
			\subseteq\prod_\mathcal{U}\Ker {\diff}^{n+1}_k
			\subseteq\Ker{\diff}^{n+1}_\mathcal{U}.
		\end{equation}
		If $\lim_\mathcal{U}\mathscr{C}({\diff^n_k})=+\infty$, then the inclusion $\overline{\Imm{\diff}^{n}_\mathcal{U}}\subseteq\prod_\mathcal{U}\overline{\Imm{\diff}^{n}_k}$ is strict by the quantitative Bader--Sauer Lemma (\ref{lemm:bader-sauer-gen}); if $\lim_\mathcal{U}\mathscr{C}({\diff^{n+1}_k})=+\infty$, then the inclusion $\prod_\mathcal{U}\Ker {\diff}^{n+1}_k\subseteq\Ker{\diff}^{n+1}_\mathcal{U}$ is strict by the same lemma.
		In both cases, the chain of inclusions $(\ref{eq:chain-incl-gen})$ shows that
		\[      
			\overline{\Imm{\diff}^{n}_\mathcal{U}}\subsetneq\Ker{\diff}^{n+1}_\mathcal{U},
		\]
		which means that $\bar{H}_a^n\left(\Gamma;{E}\right)\neq0$.
	\end{proof}

    We are ready to prove Theorem \ref{thrm:charact-vanish-asymp-gen}.
    
    \begin{mthrm}{A}
		\label{thrm:charact-vanish-asymp-gen}
        Let $\Gamma$ be a discrete group and $\mathcal{E}\coloneqq(E_k)_{k\in\mathbb{N}}$ be an asymptotic Banach $\Gamma$-module.
   		For fixed $n\in\mathbb{N}$, the following assertions are equivalent:
		\begin{enumerate}
			\item\label{thrm:charact-gen:hbara} $\bar{H}^n_a(\Gamma;\mathcal{E})=0$.
			\item\label{thrm:charact-gen:ha} $H^n_a(\Gamma;\mathcal{E})=0$.
			\item\label{thrm:charact-gen:ha-zero-hsdrff} $H^n_a(\Gamma;\mathcal{E})=0$ and $H^{n+1}_a(\Gamma;\mathcal{E})$ is Hausdorff.
		\end{enumerate}
	\end{mthrm}
	
	\begin{proof}
        The implications $\ref{thrm:charact-gen:ha-zero-hsdrff}\Rightarrow\ref{thrm:charact-gen:ha}\Rightarrow\ref{thrm:charact-gen:hbara}$ are clear; it remains to prove $\ref{thrm:charact-gen:hbara}\Rightarrow\ref{thrm:charact-gen:ha-zero-hsdrff}$.
		Assume that $\bar{H}^n_a(\Gamma;\mathcal{E})=0$.
		By Lemma \ref{lemm:vanish-hausdorff-gen}, we know that $\mathcal{E}$ has asymptotically Hausdorff $n$-bounded cohomology, and asymptotically Hausdorff $(n+1)$-bounded cohomology.
		Hence, Lemma \ref{lemm:hausdorff-ha-hb-gen} implies that both $H^n_a(\Gamma;\mathcal{E})$ and $H^{n+1}_a(\Gamma;\mathcal{E})$ are Hausdorff.
        Since $\bar{H}^n_a(\Gamma;\mathcal{E})=0$, it follows that $H^n_a(\Gamma;\mathcal{E})=0$ and \ref{thrm:charact-gen:ha-zero-hsdrff} holds.
	\end{proof}

    \subsection{Asymptotic \texorpdfstring{$n\frac{1}{2}$}{n-and-a-half} criterion}
    \label{subsec:nhalf}
    
    In \cite{glmr}*{Proposition 4.2.9}, it is proved that, if a group $\Gamma$ satisfies the \emph{$2\frac{1}{2}$-condition} --- namely $H^2_b(\Gamma;\mathbb{R})=0$ and $H^3_b(\Gamma;\mathbb{R})$ is Hausdorff --- then $H^2_a(\Gamma;\mathbb{R})=0$.
   	For the purpose of \cite{glmr}, the authors are only interested in degree $2$ and in $\mathbb{R}$-coefficients (or more generally in trivial dual coefficients), but their argument works for all degrees and for general asymptotic Banach modules, after replacing the $2\frac{1}{2}$-condition with an asymptotic generalisation; this is the content of the following proposition.
    
    \begin{prop}[Asymptotic $n\frac{1}{2}$-criterion]
        \label{prop:nnp1}
        Let $\Gamma$ be a discrete group, $\mathcal{E}\coloneqq(E_k)_{k\in\mathbb{N}}$ an asymptotic Banach $\Gamma$-module, and $n\in\mathbb{N}$.
        Suppose that
        \begin{enumerate}
            \item\label{prop:nnp1:assump1} $\mathcal{E}$ has asymptotically vanishing $n$-bounded cohomology, and
            \item\label{prop:nnp1:assump2} $\mathcal{E}$ has asymptotically Hausdorff $(n+1)$-bounded cohomology.
        \end{enumerate}
        Then $H^n_a(\Gamma;\mathcal{E})=0$.
    \end{prop}

    \begin{proof}
        Assumption \ref{prop:nnp1:assump1} says in particular that, for every $k$ in a $\mathcal{U}$-large set, the sequence
        \[
            C^{n-1}_{b,\inhom}\left(\Gamma;E_k\right)
            \xrightarrow{{\diff^n_k}}C^n_{b,\inhom}\left(\Gamma;E_k\right)
            \xrightarrow{{\diff^{n+1}_k}}C^{n+1}_{b,\inhom}\left(\Gamma;E_k\right)
        \]
        is exact.
        Moreover, \ref{prop:nnp1:assump1} and \ref{prop:nnp1:assump2} give $\lim_\mathcal{U}\mathscr{C}({\diff^n_k}),\lim_\mathcal{U}\mathscr{C}({\diff^{n+1}_k})<+\infty$.
        It follows from Lemma \ref{lemm:ultraprod-exact} that the induced sequence of ultraproducts
        \[
            \prod_\mathcal{U}C^{n-1}_{b,\inhom}\left(\Gamma;E_k\right)
            \xrightarrow{{\diff^n_\mathcal{U}}}\prod_\mathcal{U}C^n_{b,\inhom}\left(\Gamma;E_k\right)
            \xrightarrow{{\diff^{n+1}_\mathcal{U}}}\prod_\mathcal{U}C^{n+1}_{b,\inhom}\left(\Gamma;E_k\right)
        \]
        is exact.
        This means precisely that $H^n_a(\Gamma;\mathcal{E})=0$.
    \end{proof}

    \subsection{Diagonal coefficients}
    \label{subsec:diag}

    In the diagonal case $\mathcal{E}=(E)_{k\in\mathbb{N}}$ given by the data of a single Banach $\Gamma$-module $E$, we can refine Theorem \ref{thrm:charact-vanish-asymp-gen} and obtain a full characterisation of vanishing of asymptotic cohomology in terms of bounded cohomology.

    \begin{mthrm}{B}
        \label{thrm:vanish-diag}
   		Let $\Gamma$ be a discrete group, $E$ a Banach $\Gamma$-module.
   		For $n\in\mathbb{N}$, the following assertions are equivalent:
		\begin{enumerate}
			\item\label{thrm:vanish-diag:ha} $H^n_a(\Gamma;E)=0$.
			\item\label{thrm:vanish-diag:hb} $H^n_b(\Gamma;E)=0$ and $H^{n+1}_b(\Gamma;E)$ is Hausdorff.
		\end{enumerate}
    \end{mthrm}
    \begin{proof}
        \begin{description}
            \item[$\ref{thrm:vanish-diag:hb}\Rightarrow\ref{thrm:vanish-diag:ha}$]
            Condition \ref{thrm:vanish-diag:hb} is equivalent to the diagonal asymptotic Banach $\Gamma$-module $\mathcal{E}=(E)_{k\in\mathbb{N}}$ having asymptotically vanishing $n$-bounded cohomology and asymptotically Hausdorff $(n+1)$-bounded cohomology (see Proposition \ref{prop:charact-unif-cond-diag}).
            Hence, the asymptotic $n\frac{1}{2}$-criterion (Proposition \ref{prop:nnp1}) implies that $H^n_a(\Gamma;E)=0$.
            \item[$\ref{thrm:vanish-diag:ha}\Rightarrow\ref{thrm:vanish-diag:hb}$]
            If $H^n_a(\Gamma;E)=0$, then Lemma \ref{lemm:vanish-hausdorff-gen} shows that $\mathcal{E}=(E)_{k\in\mathbb{N}}$ has asymptotically Hausdorff $n$- and $(n+1)$-bounded cohomology.
            In other words, $H^n_b(\Gamma;E)$ and $H^{n+1}_b(\Gamma;E)$ are Hausdorff (see Proposition \ref{prop:charact-unif-cond-diag}).
            It remains to show that $H^n_b(\Gamma;E)=0$.
            To do so, start with an $n$-cocycle $z_0\in\Ker {\diff}^{n+1}\subseteq C_{b,\inhom}^n(\Gamma;E)$; we want to show that $z_0\in\Imm {\diff}^n$.
    		Consider the asymptotic $n$-cocycle $z\coloneqq[z_0]\in\Ker {\diff}^{n+1}_\mathcal{U}$ defined by the constant sequence $(z_0)_{k\in\mathbb{N}}$.
    		We know that $H^n_a(\Gamma;E)=0$ by assumption; this implies that $z\in\Imm{\diff}^{n}_\mathcal{U}$: there exists $c\in C_{a,\inhom}^{n-1}(\Gamma;E)=C_{b,\inhom}^{n-1}(\Gamma;E)^\mathcal{U}$ such that ${\diff}^n_\mathcal{U}\left(c\right)=z$.
    		Write $c=[c_k]$, with $(c_k)_{k\in\mathbb{N}}$ a bounded sequence in $C_{b,\inhom}^{n-1}(\Gamma;E)$.
    		Hence, we have
    		\[
    			\lim_\mathcal{U}\left\|{\diff}^n\left(c_k\right)-z_0\right\|=0.
    		\]
    		In particular, $z_0\in\overline{\Imm {\diff}^n}$.
    		But we know that $H^n_b(\Gamma;E)$ is Hausdorff, so $\Imm {\diff}^n$ is closed (see Proposition \ref{prop:charact-hausd}), and $z_0\in\Imm {\diff}^n$.
    		This concludes the proof that $H^n_b(\Gamma;E)=0$.\qedhere
        \end{description}
    \end{proof}

    \begin{remk}
        \begin{enumerate}
            \item Matsumoto and Morita proved that, if $\bar{H}^n_b(\Gamma;\mathbb{R})=0$ and $H^{n+1}_b(\Gamma;\mathbb{R})$ is Hausdorff, then the $\ell^1$-homology $H_n^{\ell^1}(\Gamma;\mathbb{R})$ vanishes \cite{matsumoto-morita}*{Corollary 2.4(ii)}.
            Hence, it follows from Theorem \ref{thrm:vanish-diag} that vanishing of $H^n_a(\Gamma;\mathbb{R})$ implies vanishing of $H_n^{\ell^1}(\Gamma;\mathbb{R})$ for any $n\in\mathbb{N}$.
            \item Conditions of the form ``$H^n=0$ and $H^{n+1}$ is Hausdorff'' also occur in unitary cohomology; see for instance work of Bader and Nowak \cite{bader-nowak:group-alg} where it is shown that such a condition holding for every unitary representation is equivalent to the Laplacian in degree $n$ having an algebraic spectral gap à la Ozawa \cite{ozawa}.
        \end{enumerate}
    \end{remk}
	
\section{Low degrees}
    \label{sec:low-deg}
	
	It is well-known that $H^1_b(\Gamma;E)=0$ for $E$ a reflexive Banach $\Gamma$-module (this was originally observed by Johnson \cite{johnson}*{Theorem 3.4}, see also \cite{monod}*{Proposition 6.2.1}).
	Analogously, we prove a vanishing result for asymptotic cohomology in degree $1$; this will allow us to deduce Hausdorffness of bounded cohomology in degree $2$ using Theorem \ref{thrm:vanish-diag}.

    Similar to the argument of \cite{bdlhv}*{Proposition 2.2.9} for bounded cohomology with unitary coefficients, our proof goes via the Chebyshev centre of a bounded set in a uniformly convex Banach space.
    The key point that allows us to adapt this to asymptotic cohomology is a continuity result for the Chebyshev centre with respect to the Hausdorff distance --- this is Lemma \ref{lemm:centre}\ref{lemm:centre:cont}.

    We first recall some notions around uniform convexity and reflexivity in \S{}\ref{subsec:bgnd-unif-conv}; then we prove what we need about the Chebyshev centre in \S{}\ref{subsec:chebyshev}.
    This leads in \S{}\ref{subsec:vanish-deg1} to vanishing of asymptotic cohomology in degree $1$: this is Theorem \ref{thrm:vanishing-h1a}.
    Finally, in \S{}\ref{subsec:hausdorff-deg2}, we deduce Hausdorffness of bounded cohomology in degree $2$: this is Corollary \ref{coro:h2b-hausdorff}.

    \subsection{Background on uniform convexity}\label{subsec:bgnd-unif-conv}

    We briefly recall some notions around uniform convexity and reflexivity.

    Let $\delta:(0,2]\rightarrow(0,1]$ be a continuous, non-decreasing function.
    We say that a collection $\mathcal{E}\coloneqq(E_i)_{i\in I}$ of normed space is \emph{uniformly $\delta$-convex}, or that $\delta$ is a \emph{modulus of convexity} for $\mathcal{E}$, if for all $i\in I$, for all $\varepsilon\in(0,2]$, and for all $u_1,u_2\in E_i$ with $\|u_1\|,\|u_2\|\leq1$ we have
    \[
        \left\|u_1-u_2\right\|\geq\varepsilon\Rightarrow\left\|\frac{u_1+u_2}{2}\right\|\leq 1-\delta(\varepsilon).
    \]
    We say that $\mathcal{E}$ is \emph{uniformly convex} if it is uniformly $\delta$-convex for some continuous, non-decreasing function $\delta:(0,2]\rightarrow(0,1]$.

    For a single normed space $E$, this coincides with the usual notion of uniform convexity (see \cite{gurarij} regarding continuity of $\delta$).

    Recall that the \emph{weak topology} on a normed space $E$ is the coarsest topology such that every bounded linear map $f:E\rightarrow\mathbb{R}$ is continuous.
    In other words, a sequence $(x_k)_{k\in\mathbb{N}}$ in $E$ converges in the weak topology, or \emph{weakly converges}, to an element $x_\infty\in E$ if and only if, for every bounded linear map $f:E\rightarrow\mathbb{R}$, we have $f(x_k)\to f(x_\infty)$ as $k\to\infty$.

    We will use the Milman--Pettis Theorem \cites{milman,pettis}, which asserts that, if $E$ is a uniformly convex Banach space, then $E$ is reflexive; equivalently, convex closed bounded sets in $E$ are weakly compact \cite{bourbaki}*{Chapitre 4, \S{}2.4, Proposition 6}.

    \subsection{The Chebyshev centre}
    \label{subsec:chebyshev}

    Given a normed space $E$, and two non-empty bounded subsets $X_1,X_2$, recall that the \emph{Hausdorff distance} from $X_1$ to $X_2$ is
    \[
        \dHaus(X_1,X_2)\coloneqq\max\left\{\sup_{x_1\in X_1}\left(\inf_{x_2\in X_2}\|x_1-x_2\|\right),\sup_{x_2\in X_2}\left(\inf_{x_1\in X_1}\|x_2-x_1\|\right)\right\}.
    \]
    
    \begin{lemm}[Chebyshev centre]
        \label{lemm:centre}
        Let $\delta:(0,2]\rightarrow(0,1]$ be a continuous, non-decreasing function, and let $E$ be a uniformly $\delta$-convex Banach space.
        \begin{enumerate}
            \item\label{lemm:centre:exist} \emph{(Existence and uniqueness)} Let $X$ be a non-empty bounded subset of $E$.
            The function $\maxdist_X:E\rightarrow[0,+\infty)$ given by
            \[
                \maxdist_X:y\mapsto\sup_{x\in X}\|x-y\|
            \]
            has a unique minimum, attained at a point called the \emph{Chebyshev centre} of $X$, and denoted by $\omega(X)$.
            The quantity $\rho(X)\coloneqq\maxdist_X(\omega(X))$ is called the \emph{Chebyshev radius} of $X$.
            \item\label{lemm:centre:indist} \emph{(Indistinguishable sets)} If $X,X'$ are two non-empty bounded subsets of $E$ with $\dHaus(X,X')=0$, then $\omega(X)=\omega(X')$.
            \item\label{lemm:centre:cont} \emph{(Continuity)} There is a function $F:[0,+\infty)^2\rightarrow[0,\infty)$ depending only on $\delta$ such that, for all non-empty bounded subsets $X,X'$ of $E$ satisfying $\rho(X)=\rho(X')\eqqcolon\rho$, and for all $\Delta\geq\dHaus(X,X')$, we have
            \[
                \|\omega(X)-\omega(X')\|\leq F\left(\Delta,\rho\right).
            \]
            Moreover, if $(\Delta_k)_{k\in\mathbb{N}},(\rho_k)_{k\in\mathbb{N}}$ are two bounded sequences in $[0,+\infty)$ with $\Delta_k\to0$ as $k\to\infty$, then
            \[
                \lim_{k\to\infty}F(\Delta_k,\rho_k)=0.
            \]
        \end{enumerate}
    \end{lemm}
    \begin{proof}    
        Let $X$ be a non-empty bounded subset of $E$.
        Let $\rho\coloneqq\inf_{y\in E}\maxdist_X(y)$, and for each $t>\rho$, let
        \[
            C(t)\coloneqq\left\{y\in E\::\:\maxdist_X(y)\leq t\right\}.
        \]
        The set $C(t)$ is non-empty by choice of $\rho$, convex, closed, and bounded.
        By uniform convexity, $C(t)$ is weakly compact.
        Moreover, we have $C(s)\subseteq C(t)$ whenever $\rho<s\leq t$, so the set
        \[
            C\coloneqq\bigcap_{t>\rho}C(t)
        \]
        is non-empty, convex, and weakly compact.
        The following claim will give us an upper bound on the diameter of each $C(t)$.
        \begin{claim*}
            Let $t>\rho$, and let $\varepsilon\in(0,2]$ such that $\delta(\varepsilon)>1-\rho/t$.
            Then
            \[
                \forall y_1,y_2\in C(t),\:\|y_1-y_2\|<\varepsilon t.
            \]
        \end{claim*}
        \begin{proof}[Proof of the claim]
            Let $y_1,y_2\in C(t)$, and suppose for contradiction that $\|y_1-y_2\|\geq\varepsilon t$.
            Consider the midpoint ${m}\coloneqq\frac{1}{2}(y_1+y_2)$; it lies in $C(t)$ by convexity.
            Moreover, for all $x\in X$, we have
            \[
                \|x-{m}\|=t\left\|\frac{1}{2}\left(\frac{x-y_1}{t}+\frac{x-y_2}{t}\right)\right\|.
            \]
            Setting $u_i\coloneqq\frac{1}{t}(x-y_i)$ for $i=1,2$, we have $\|u_i\|\leq 1$ and $\|u_1-u_2\|\geq\varepsilon$.
            By definition of the modulus of convexity, it follows that
            \[
                \|x-{m}\|\leq t\left(1-\delta(\varepsilon)\right).
            \]
            Since this is true for all $x\in X$, we have
            \[
                \maxdist_X({m})\leq t\left(1-\delta(\varepsilon)\right).
            \]
            But $t\left(1-\delta(\varepsilon)\right)<\rho$ by choice of $\varepsilon$; this contradicts the definition of $\rho$.
        \end{proof}

        It follows from the claim that $C$ is a singleton.
        Indeed, suppose for contradiction that $C$ contains two distinct points $y_1\neq y_2$, and choose $\varepsilon\in(0,2]$ such that $\varepsilon(\rho+1)<\|y_1-y_2\|$.
        Since $\delta(\varepsilon)>0$, and we can choose $t\in(\rho,\rho+1]$ such that $1-\rho/t<\delta(\varepsilon)$.
        Now we have $y_1,y_2\in C\subseteq C(t)$, but $\|y_1-y_2\|\geq\varepsilon t$, contradicting the claim.
        Therefore, $C$ is a singleton, which proves \ref{lemm:centre:exist}.

        For \ref{lemm:centre:indist}, the above definition says that $\omega(X)$ is the centre of the unique closed ball of minimal radius containing $X$.
        But note that sets $X,X'$ with $\dHaus(X,X')=0$ have the property that, whenever $\bar{B}$ is a closed ball in $E$, we have $X\subseteq\bar{B}$ if and only if $X'\subseteq\bar{B}$.
        It follows that $\omega(X)=\omega(X')$.

        For \ref{lemm:centre:cont}, consider a second non-empty bounded subset $X'$ of $E$ and suppose that $\rho(X)=\rho(X')\eqqcolon\rho$.
        Let $\Delta\geq\dHaus(X,X')$; we assume for now that $\Delta>0$.
        For every $x\in X$, and for every $\epsilon>0$ there exists $x'\in X'$ with $\|x-x'\|\leq\Delta+\epsilon$.
        Therefore,
        \begin{align*}
            \|x-\omega(X')\|
            \leq\|x-x'\|+\|x'-\omega(X')\|
            \leq\Delta+\epsilon+\rho.
        \end{align*}
        Since this holds for all $x\in X$ and for all $\epsilon>0$, we have
        \begin{equation*}
            \maxdist_{X}\left(\omega(X')\right)\leq\rho+\Delta.
        \end{equation*}
        Keeping the same notations as above, this tells us that
        \[
            \omega(X')\in C\left(\rho+\Delta\right),
        \]
        and also $\omega(X)\in C\left(\rho+\Delta\right)$ by definition.
        Define $\varepsilon:(0,+\infty)\times[0,+\infty)\rightarrow[0,+\infty)$ by
        \[
            \varepsilon\left(\Delta,\rho\right)\coloneqq\inf\left\{\varepsilon\in(0,2]\::\:\delta(\varepsilon)\geq1-\left(\frac{\rho}{\rho+\Delta}\right)^2\right\}\in[0,+\infty).
        \]
        Hence, we have
        \[
            \delta\left(\varepsilon\left(\Delta,\rho\right)\right)
            \geq1-\left(\frac{\rho}{\rho+\Delta}\right)^2>1-\frac{\rho}{\rho+\Delta},
        \]
        as we assumed that $\Delta>0$.
        Since both $\omega(X),\omega(X')$ lie in $C(\rho+\Delta)$, the previous claim now implies that
        \begin{equation}
            \label{eq:cont-centre}
            \left\|\omega(X)-\omega(X')\right\|\leq F\left(\Delta,\rho\right),
        \end{equation}
        where
        \[
            F(\Delta,\rho)\coloneqq\varepsilon(\Delta,\rho)\cdot(\rho+\Delta).
        \]
        Setting $\varepsilon(0,\rho)=0$ and thus $F(0,\rho)=0$ for all $\rho\in[0,+\infty)$, item \ref{lemm:centre:indist} implies that the inequality $(\ref{eq:cont-centre})$ still holds when $\Delta=\dHaus(X,X')=0$, and is therefore true in general.

        It remains to prove the final statement: let $(\Delta_k)_{k\in\mathbb{N}},(\rho_k)_{k\in\mathbb{N}}$ be bounded sequences in $[0,+\infty)$ with $\Delta_k\to0$ as $k\to\infty$.
        Hence, $1-\frac{\rho_k}{\rho_k+\Delta_k}\to0$.
        Since $\delta$ is continuous and non-decreasing, it follows that $\varepsilon(\Delta_k,\rho_k)\to0$.
        The sequences $(\Delta_k)_{k\in\mathbb{N}},(\rho_k)_{k\in\mathbb{N}}$ are bounded, so
        \[
            F(\Delta_k,\rho_k)=\varepsilon(\Delta_k,\rho_k)\cdot(\rho_k+\Delta_k)\xrightarrow[k\to\infty]{}0.\qedhere
        \]
    \end{proof}

    \subsection{Vanishing of asymptotic cohomology in degree 1}
    \label{subsec:vanish-deg1}

    An asymptotic $\Gamma$-module $\mathcal{E}\coloneqq(E_k)_{k\in\mathbb{N}}$ is said to be \emph{uniformly convex} if the collection $(E_k)_{k\in\mathbb{N}}$ of Banach spaces is uniformly convex in the sense of \S{}\ref{subsec:bgnd-unif-conv}.

    \begin{mthrm}{C}
        \label{thrm:vanishing-h1a}
        Let $\Gamma$ be a discrete group and let $\mathcal{E}\coloneqq(E_k)_{k\in\mathbb{N}}$ be a uniformly convex asymptotic Banach $\Gamma$-module.
        Then $H_a^1(\Gamma;\mathcal{E})=0$.
    \end{mthrm}
    \begin{proof}
        Let $z\in\Ker\left({\diff}^2_\mathcal{U}:C^1_{a,\inhom}(\Gamma;\mathcal{E})\rightarrow C^2_{a,\inhom}(\Gamma;\mathcal{E})\right)$ be an asymptotic $1$-cocycle.
        We want to show that $z$ is an asymptotic coboundary.
		Recall that $C^1_{a,\inhom}(\Gamma;\mathcal{E})=\prod_\mathcal{U}C^1_{b,\inhom}(\Gamma;E_k)$ and pick a bounded sequence $(z_k)_{k\in\mathbb{N}}\in\prod_{k\in\mathbb{N}}^{\ell^\infty}C^1_{b,\inhom}(\Gamma;E_k)$ with $z=[z_k]$.

        For each $k\in\mathbb{N}$, consider the set
        \[
            X_k\coloneqq\left\{z_k(\gamma)\::\:\gamma\in\Gamma\right\}\subseteq E_k.
        \]
        By boundedness of $z_k$, the set $X_k$ is bounded.
        We will consider the Chebyshev centre $\omega(X_k)$ of $X_k$, which exists by Lemma \ref{lemm:centre}\ref{lemm:centre:exist}.

        We want to bound $\|\gamma\cdot\omega(X_k)-\omega(X_k)+z_k(\gamma)\|$, for $\gamma\in\Gamma$.
        To do so, note that, for $\gamma,\gamma'\in\Gamma$, we have
        \[
            \gamma z_k(\gamma')=z_k(\gamma\gamma')-z_k(\gamma)+{\diff^2}(z_k)(\gamma,\gamma').
        \]
        It follows easily that
        \[
            \dHaus\left(\gamma X_k,X_k-z_k(\gamma)\right)\leq\left\|{\diff^2}(z_k)\right\|,
        \]
        where $X_k-z_k(\gamma)\coloneqq\{x-z_k(\gamma)\::\:x\in X_k\}$.
        Noting that $\omega(\gamma X_k)=\gamma\cdot\omega(X_k)$ and $\omega(X_k-z_k(\gamma))=\omega(X_k)-z_k(\gamma)$ (as both $\gamma$ and translation by $z_k(\gamma)$ are affine isometries), and similarly for the Chebyshev radius $\rho$, Lemma \ref{lemm:centre}\ref{lemm:centre:cont} gives a function $F:[0,+\infty)^2\rightarrow[0,+\infty)$ such that
        \begin{equation}
            \label{eq:asymp-cobound}
            \|\gamma\cdot\omega(X_k)-\omega(X_k)+z_k(\gamma)\|\leq F\left(\left\|{\diff^2}(z_k)\right\|,\rho(X_k)\right).
        \end{equation}
        Crucially, $F$ depends only on the modulus of convexity, and is thus uniform in $k$.

        For each $k$, we have a $0$-cochain $\omega(X_k)\in C^0_{b,\inhom}(\Gamma;E_k)$.
        The sequence $(z_k)_{k\in\mathbb{N}}$ was assumed to be bounded, so the $X_k$'s are all contained in a ball of common radius in their respective spaces, and the sequence $(\omega(X_k))_{k\in\mathbb{N}}$ is also bounded; it defines an asymptotic $0$-cochain
        \[
            c\coloneqq\left[\omega(X_k)\right]\in C^0_{a,\inhom}(\Gamma;\mathcal{E})=\textstyle\prod_\mathcal{U}C^0_{b,\inhom}(\Gamma;E_k).
        \]
        Finally, we compute
        \begin{align*}
            \left\|{\diff^1_\mathcal{U}}(c)+z\right\|
            &=\lim_{\mathcal{U}}\left\|{\diff^1}\left(\omega(X_k)\right)+z_k\right\|
            \\&=\lim_\mathcal{U}\sup_{\gamma\in\Gamma}\left\|\gamma\cdot\omega(X_k)-\omega(X_k)+z_k(\gamma)\right\|
            \\&\leq\lim_\mathcal{U} F\left(\left\|{\diff^2}(z_k)\right\|,\rho(X_k)\right)
            &\textrm{by $(\ref{eq:asymp-cobound})$}.
        \end{align*}
        But $\lim_\mathcal{U}\|{\diff^2}(z_k)\|=\|{\diff^2_\mathcal{U}(z)}\|=0$, and $(\rho(X_k))_{k\in\mathbb{N}}$ is bounded by boundedness of $(z_k)_{k\in\mathbb{N}}$, so the last part of Lemma \ref{lemm:centre}\ref{lemm:centre:cont} implies that $\lim_\mathcal{U} F(\|{\diff^2}(z_k)\|,\rho(X_k))=0$, and therefore
        \[
            z=-{\diff^1_\mathcal{U}}(c),
        \]
        so $z$ is an asymptotic coboundary as wanted.
    \end{proof}

    \begin{remk}
        In the case where each $E_k$ has trivial $\Gamma$-action, there is an easier proof, not using the Chebyshev centre, that $H^1_a(\Gamma;\mathcal{E})=0$.

        Let $z\in\Ker\left({\diff}^2_\mathcal{U}:C^1_{a,\inhom}(\Gamma;\mathcal{E})\rightarrow C^2_{a,\inhom}(\Gamma;\mathcal{E})\right)$ be an asymptotic $1$-cocycle.
        We will show that $z=0$; in particular, $z$ is an asymptotic coboundary.
		To do so, pick a bounded sequence $(z_k)_{k\in\mathbb{N}}\in\prod_{k\in\mathbb{N}}^{\ell^\infty}C^1_{b,\inhom}(\Gamma;E_k)$ with $z=[z_k]$.

        For all $\gamma\in\Gamma$ and $k,m\in\mathbb{N}$, a simple computation (using the fact that $\Gamma$ acts trivially on $E_k$) shows that
        \[
            z_k(\gamma)=\frac{1}{m}z_k(\gamma^m)+\frac{1}{m}\sum_{i=1}^{m-1}{\diff^2}(z_k)\left(\gamma^i,\gamma\right).
        \]
        (The trick of rewriting $z_k$ as above is Ivanov's \cite{ivanov}.)
        Taking $m=k$ and passing to the $\ell^\infty$-norm, we deduce that
        \[
            \|z_k\|\leq\frac{1}{k}\|z_k\|+\|{\diff}^2(z_k)\|.
        \]
        But the assumption that ${\diff}^2_\mathcal{U}(z)=0$ means that $\lim_\mathcal{U}\|{\diff}^2(z_k)\|=0$; moreover, the sequence $(\|z_k\|)_{k\in\mathbb{N}}$ is bounded.
        Therefore,
        \[
            \|z\|_\mathcal{U}=\lim_\mathcal{U}\|z_k\|=0.\qedhere
        \]
    \end{remk}

    \subsection{Hausdorffness of bounded cohomology in degree 2}
    \label{subsec:hausdorff-deg2}

    Theorem \ref{thrm:vanishing-h1a}, together with our characterisation of vanishing of asymptotic cohomology, now immediately implies the following.
    
    \begin{mcoro}{D}
        \label{coro:h2b-hausdorff}
        Let $\Gamma$ be a discrete group and let $E$ be a uniformly convex Banach $\Gamma$-module.
        Then $H_b^2(\Gamma;E)$ is Hausdorff.
    \end{mcoro}
    \begin{proof}
        We have $H^1_a(\Gamma;E)=H^1_a(\Gamma;(E)_{k\in\mathbb{N}})=0$ by Theorem \ref{thrm:vanishing-h1a}, so $H_b^2(\Gamma;E)$ is Hausdorff by Theorem \ref{thrm:vanish-diag}.
    \end{proof}

    \begin{remk}
        Hausdorffness of $H^2_b(\Gamma;E)$ was previously known in the following cases:
        \begin{itemize}
            \item For $\Gamma$ discrete and $E=\mathbb{R}$ with trivial $\Gamma$-action \cites{matsumoto-morita,ivanov},
            \item For $\Gamma$ finitely generated (or more generally, compactly generated locally compact second countable) and $E$ separable \cite{monod}*{Corollary 11.4.2}.
        \end{itemize}
    \end{remk}

    \begin{remk}
        Hausdorffness of bounded cohomology is known not to hold in higher degrees, even with trivial real coefficients.
        \begin{itemize}
            \item In degree $3$, it was proved by Franceschini et al. \cite{ffps} that $H^3_b(\Gamma;\mathbb{R})$ is never Hausdorff when $\Gamma$ is acylindrically hyperbolic (extending results of Soma for free groups \cite{soma:free} and surface groups \cite{soma:srf}).
            \item In any degree $n\geq5$, it was proved by Soma \cite{soma:free} that there exist finitely generated, discrete groups $\Gamma$ such that $H^n_b(\Gamma;\mathbb{R})$ is not Hausdorff.
        \end{itemize}        
    \end{remk}
    
\bibliography{Cohomology-refs.bib}

@unpublished{bader-sauer:below-rk,
	title = {{Higher property T and below-rank phenomena of lattices}},
	author = {Bader, Uri and Sauer, Roman},
	note = {Preprint available on arXiv:\href{https://arxiv.org/abs/2511.20192}{2511.20192}},
}

@unpublished{bader-sauer:unitary,
	title = {{Higher Kazhdan property and unitary cohomology of arithmetic groups}},
	author = {Bader, Uri and Sauer, Roman},
	note = {Preprint available on arXiv:\href{https://arxiv.org/abs/2308.06517}{2308.06517}},
}

@unpublished{bogliolo,
	title = {{Ulam stability of groups with displacement properties}},
	author = {Bogliolo, Elena},
	note = {In preparation},
}

@unpublished{glmr,
	title = {{Asymptotic Cohomology and Uniform Stability for Lattices in Semisimple Groups}},
	author = {Glebsky, Lev and Lubotzky, Alexander and Monod, Nicolas and Rangarajan, Bharatram},
	note = {Preprint available on arXiv:\href{https://arxiv.org/abs/2301.00476}{2301.00476}},
}

@book {banach,
    AUTHOR = {Banach, Stefan},
     TITLE = {Th\'eorie des op\'erations lin\'eaires},
      NOTE = {Reprint of the 1932 original},
 PUBLISHER = {\'Editions Jacques Gabay, Sceaux},
      YEAR = {1993},
     PAGES = {iv+128},
      ISBN = {2-87647-148-5},
   MRCLASS = {01A75 (46-03 47-03)},
  MRNUMBER = {1357166},
}

@book {bdlhv,
    AUTHOR = {Bekka, Bachir and de la Harpe, Pierre and Valette, Alain},
     TITLE = {Kazhdan's property ({T})},
    SERIES = {New Mathematical Monographs},
    VOLUME = {11},
 PUBLISHER = {Cambridge University Press, Cambridge},
      YEAR = {2008},
     PAGES = {xiv+472},
      ISBN = {978-0-521-88720-5},
   MRCLASS = {22-02 (22E40 28D15 37A15 43A07 43A35)},
  MRNUMBER = {2415834},
MRREVIEWER = {Markus\ Neuhauser},
       DOI = {10.1017/CBO9780511542749},
       URL = {https://doi.org/10.1017/CBO9780511542749},
}

@book {bourbaki,
    AUTHOR = {Bourbaki, Nicolas},
     TITLE = {Espaces vectoriels topologiques. {C}hapitres 1 \`a{} 5},
   EDITION = {New},
      NOTE = {\'El\'ements de math\'ematique. [Elements of mathematics]},
 PUBLISHER = {Masson, Paris},
      YEAR = {1981},
     PAGES = {vii+368},
      ISBN = {2-225-68410-3},
   MRCLASS = {46-02 (46-01 46Axx 47D15)},
  MRNUMBER = {633754},
MRREVIEWER = {E.\ Gerlach},
}

@book {abramovich-aliprantis,
    AUTHOR = {Abramovich, Y. A. and Aliprantis, C. D.},
     TITLE = {An invitation to operator theory},
    SERIES = {Graduate Studies in Mathematics},
    VOLUME = {50},
 PUBLISHER = {American Mathematical Society, Providence, RI},
      YEAR = {2002},
     PAGES = {xiv+530},
      ISBN = {0-8218-2146-6},
   MRCLASS = {47B60 (46-01 46B42 46Bxx 47-01 47Axx)},
  MRNUMBER = {1921782},
MRREVIEWER = {Anthony\ W.\ Wickstead},
       DOI = {10.1090/gsm/050},
       URL = {https://doi.org/10.1090/gsm/050},
}

@article {bader-nowak:group-alg,
    AUTHOR = {Bader, Uri and Nowak, Piotr W.},
     TITLE = {Group algebra criteria for vanishing of cohomology},
   JOURNAL = {J. Funct. Anal.},
  FJOURNAL = {Journal of Functional Analysis},
    VOLUME = {279},
      YEAR = {2020},
    NUMBER = {11},
     PAGES = {108730, 18},
      ISSN = {0022-1236,1096-0783},
   MRCLASS = {20C07 (22D55 46L05 46M20)},
  MRNUMBER = {4141491},
MRREVIEWER = {Alexander\ Isaakovich\ Shtern},
       DOI = {10.1016/j.jfa.2020.108730},
       URL = {https://doi.org/10.1016/j.jfa.2020.108730},
}

@proceedings {campagnolo:hb,
     TITLE = {Bounded cohomology and simplicial volume},
    SERIES = {London Mathematical Society Lecture Note Series},
    VOLUME = {479},
 BOOKTITLE = {Proceedings of the {I}nternational {Y}oung {S}eminar on
              {B}ounded {C}ohomology and {S}implical {V}olume held online
              {N}ovember 2020--{F}ebruary 2021},
    EDITOR = {Campagnolo, Caterina and Fournier-Facio, Francesco and Heuer,
              Nicolaus and Moraschini, Marco},
 PUBLISHER = {Cambridge University Press, Cambridge},
      YEAR = {2023},
     PAGES = {xxiii+146},
      ISBN = {978-1-009-18329-1},
   MRCLASS = {53C23 (53-06 57-06)},
  MRNUMBER = {4496339},
       DOI = {10.1016/j.euromechflu.2022.10.003},
       URL = {https://doi.org/10.1016/j.euromechflu.2022.10.003},
}

@article {delorme,
    AUTHOR = {Delorme, Patrick},
     TITLE = {{$1$}-cohomologie des repr\'esentations unitaires des groupes
              de {L}ie semi-simples et r\'esolubles. {P}roduits tensoriels
              continus de repr\'esentations},
   JOURNAL = {Bull. Soc. Math. France},
  FJOURNAL = {Bulletin de la Soci\'et\'e{} Math\'ematique de France},
    VOLUME = {105},
      YEAR = {1977},
    NUMBER = {3},
     PAGES = {281--336},
      ISSN = {0037-9484},
   MRCLASS = {22E30},
  MRNUMBER = {578893},
       URL = {http://www.numdam.org/item?id=BSMF_1977__105__281_0},
}

@article {fflm,
    AUTHOR = {Fournier-Facio, Francesco and L\"oh, Clara and Moraschini,
              Marco},
     TITLE = {Bounded cohomology and binate groups},
   JOURNAL = {J. Aust. Math. Soc.},
  FJOURNAL = {Journal of the Australian Mathematical Society},
    VOLUME = {115},
      YEAR = {2023},
    NUMBER = {2},
     PAGES = {204--239},
      ISSN = {1446-7887,1446-8107},
   MRCLASS = {20J06 (20F65 57T99)},
  MRNUMBER = {4640119},
MRREVIEWER = {Olympia\ Talelli},
       DOI = {10.1017/s1446788722000106},
       URL = {https://doi.org/10.1017/s1446788722000106},
}

@article {ffr,
    AUTHOR = {Fournier-Facio, Francesco and Rangarajan, Bharatram},
     TITLE = {Ulam stability of lamplighters and {T}hompson groups},
   JOURNAL = {Math. Ann.},
  FJOURNAL = {Mathematische Annalen},
    VOLUME = {389},
      YEAR = {2024},
    NUMBER = {3},
     PAGES = {2469--2497},
      ISSN = {0025-5831,1432-1807},
   MRCLASS = {20F65 (03C20 20J05 22D05 46L10)},
  MRNUMBER = {4753068},
MRREVIEWER = {Martyn\ Quick},
       DOI = {10.1007/s00208-023-02708-5},
       URL = {https://doi.org/10.1007/s00208-023-02708-5},
}

@article {ffps,
    AUTHOR = {Franceschini, Federico and Frigerio, Roberto and Pozzetti,
              Maria Beatrice and Sisto, Alessandro},
     TITLE = {The zero norm subspace of bounded cohomology of acylindrically
              hyperbolic groups},
   JOURNAL = {Comment. Math. Helv.},
  FJOURNAL = {Commentarii Mathematici Helvetici. A Journal of the Swiss
              Mathematical Society},
    VOLUME = {94},
      YEAR = {2019},
    NUMBER = {1},
     PAGES = {89--139},
      ISSN = {0010-2571,1420-8946},
   MRCLASS = {20J06 (20F65 20F67 57M07 57M50)},
  MRNUMBER = {3941468},
MRREVIEWER = {Ian\ J.\ Leary},
       DOI = {10.4171/CMH/456},
       URL = {https://doi.org/10.4171/CMH/456},
}

@book {frigerio,
    AUTHOR = {Frigerio, Roberto},
     TITLE = {Bounded cohomology of discrete groups},
    SERIES = {Mathematical Surveys and Monographs},
    VOLUME = {227},
 PUBLISHER = {American Mathematical Society, Providence, RI},
      YEAR = {2017},
     PAGES = {xvi+193},
      ISBN = {978-1-4704-4146-3},
   MRCLASS = {57T10 (20J06 37C85 55N10 57M07 57N16)},
  MRNUMBER = {3726870},
MRREVIEWER = {Clara\ L\"oh},
       DOI = {10.1090/surv/227},
       URL = {https://doi.org/10.1090/surv/227},
}

@article {gromov,
    AUTHOR = {Gromov, Michael},
     TITLE = {Volume and bounded cohomology},
   JOURNAL = {Inst. Hautes \'Etudes Sci. Publ. Math.},
  FJOURNAL = {Institut des Hautes \'Etudes Scientifiques. Publications
              Math\'ematiques},
    NUMBER = {56},
      YEAR = {1982},
     PAGES = {5--99},
      ISSN = {0073-8301,1618-1913},
   MRCLASS = {53C20 (53C21 57R99 58E99)},
  MRNUMBER = {686042},
MRREVIEWER = {Karsten\ Grove},
       URL = {http://www.numdam.org/item?id=PMIHES_1982__56__5_0},
}

@article {guichardet:dgthm,
    AUTHOR = {Guichardet, Alain},
     TITLE = {Sur la cohomologie des groupes topologiques. {II}},
   JOURNAL = {Bull. Sci. Math. (2)},
  FJOURNAL = {Bulletin des Sciences Math\'ematiques. 2e S\'erie},
    VOLUME = {96},
      YEAR = {1972},
     PAGES = {305--332},
      ISSN = {0007-4497},
   MRCLASS = {22D10},
  MRNUMBER = {340464},
MRREVIEWER = {B.\ E.\ Johnson},
}

@article {gurarij,
    AUTHOR = {Gurari\u{i}, V. I.},
     TITLE = {Differential properties of the convexity moduli of {B}anach
              spaces},
   JOURNAL = {Mat. Issled.},
  FJOURNAL = {Akademiya Nauk Moldavsko\u i\ SSR. Institut Matematiki s
              Vychislitel\cprime nym Tsentrom. Matematicheskie
              Issledovaniya},
    NUMBER = {1},
    VOLUME = {2},
      YEAR = {1967},
     PAGES = {141--148},
      ISSN = {0542-9994},
   MRCLASS = {46.10},
  MRNUMBER = {211245},
MRREVIEWER = {M.\ M.\ Day},
}

@article {ivanov,
    AUTHOR = {Ivanov, N. V.},
     TITLE = {The second bounded cohomology group},
   JOURNAL = {Zap. Nauchn. Sem. Leningrad. Otdel. Mat. Inst. Steklov.
              (LOMI)},
  FJOURNAL = {Zapiski Nauchnykh Seminarov Leningradskogo Otdeleniya
              Matematicheskogo Instituta imeni V. A. Steklova Akademii Nauk
              SSSR (LOMI)},
    VOLUME = {167},
      YEAR = {1988},
     PAGES = {117--120, 191},
      ISSN = {0373-2703},
   MRCLASS = {55N20 (20J99 53C20)},
  MRNUMBER = {964260},
MRREVIEWER = {Marek\ Golasi\'nski},
       DOI = {10.1007/BF01099246},
       URL = {https://doi.org/10.1007/BF01099246},
}

@book {johnson,
    AUTHOR = {Johnson, Barry Edward},
     TITLE = {Cohomology in {B}anach algebras},
    SERIES = {Memoirs of the American Mathematical Society},
    VOLUME = {No. 127},
 PUBLISHER = {American Mathematical Society, Providence, RI},
      YEAR = {1972},
     PAGES = {iii+96},
   MRCLASS = {46M20 (46H25)},
  MRNUMBER = {374934},
MRREVIEWER = {Alain\ Guichardet},
}

@article {loeh:anote,
    AUTHOR = {L\"oh, Clara},
     TITLE = {A note on bounded-cohomological dimension of discrete groups},
   JOURNAL = {J. Math. Soc. Japan},
  FJOURNAL = {Journal of the Mathematical Society of Japan},
    VOLUME = {69},
      YEAR = {2017},
    NUMBER = {2},
     PAGES = {715--734},
      ISSN = {0025-5645,1881-1167},
   MRCLASS = {55N35 (20J06)},
  MRNUMBER = {3638282},
MRREVIEWER = {Masaki\ Kameko},
       DOI = {10.2969/jmsj/06920715},
       URL = {https://doi.org/10.2969/jmsj/06920715},
}

@article {matsumoto-morita,
    AUTHOR = {Matsumoto, Shigenori and Morita, Shigeyuki},
     TITLE = {Bounded cohomology of certain groups of homeomorphisms},
   JOURNAL = {Proc. Amer. Math. Soc.},
  FJOURNAL = {Proceedings of the American Mathematical Society},
    VOLUME = {94},
      YEAR = {1985},
    NUMBER = {3},
     PAGES = {539--544},
      ISSN = {0002-9939,1088-6826},
   MRCLASS = {55N99 (57T99)},
  MRNUMBER = {787909},
       DOI = {10.2307/2045250},
       URL = {https://doi.org/10.2307/2045250},
}

@article{milman,
 author = {Milman, D.},
 title = {On some criteria for the regularity of spaces of the type ({B})},
 fjournal = {Comptes Rendus (Doklady) de l'Acad{\'e}mie des Sciences de l'URSS, Nouvelle S{\'e}rie},
 journal = {C. R. (Dokl.) Acad. Sci. URSS, n. Ser.},
 issn = {1819-0723},
 volume = {20},
 pages = {243--246},
 year = {1938},
 language = {English},
 zbMATH = {3031612},
 Zbl = {0019.41601}
}

@book {monod,
    AUTHOR = {Monod, Nicolas},
     TITLE = {Continuous bounded cohomology of locally compact groups},
    SERIES = {Lecture Notes in Mathematics},
    VOLUME = {1758},
 PUBLISHER = {Springer-Verlag, Berlin},
      YEAR = {2001},
     PAGES = {x+214},
      ISBN = {3-540-42054-1},
   MRCLASS = {46M20 (22D12 22E41 43A07 46H25)},
  MRNUMBER = {1840942},
MRREVIEWER = {Andrzej\ Kozlowski},
       DOI = {10.1007/b80626},
       URL = {https://doi.org/10.1007/b80626},
}

@article {monod-nariman,
    AUTHOR = {Monod, Nicolas and Nariman, Sam},
     TITLE = {Bounded and unbounded cohomology of homeomorphism and
              diffeomorphism groups},
   JOURNAL = {Invent. Math.},
  FJOURNAL = {Inventiones Mathematicae},
    VOLUME = {232},
      YEAR = {2023},
    NUMBER = {3},
     PAGES = {1439--1475},
      ISSN = {0020-9910,1432-1297},
   MRCLASS = {57S05 (20J06)},
  MRNUMBER = {4588567},
MRREVIEWER = {Bena\ Tshishiku},
       DOI = {10.1007/s00222-023-01181-w},
       URL = {https://doi.org/10.1007/s00222-023-01181-w},
}

@article {ozawa,
    AUTHOR = {Ozawa, Narutaka},
     TITLE = {Noncommutative real algebraic geometry of {K}azhdan's property
              ({T})},
   JOURNAL = {J. Inst. Math. Jussieu},
  FJOURNAL = {Journal of the Institute of Mathematics of Jussieu. JIMJ.
              Journal de l'Institut de Math\'ematiques de Jussieu},
    VOLUME = {15},
      YEAR = {2016},
    NUMBER = {1},
     PAGES = {85--90},
      ISSN = {1474-7480,1475-3030},
   MRCLASS = {22D10 (20F10 22D15 46L89)},
  MRNUMBER = {3427595},
MRREVIEWER = {Alain\ Valette},
       DOI = {10.1017/S1474748014000309},
       URL = {https://doi.org/10.1017/S1474748014000309},
}

@article {pettis,
    AUTHOR = {Pettis, B. J.},
     TITLE = {A proof that every uniformly convex space is reflexive},
   JOURNAL = {Duke Math. J.},
  FJOURNAL = {Duke Mathematical Journal},
    VOLUME = {5},
      YEAR = {1939},
    NUMBER = {2},
     PAGES = {249--253},
      ISSN = {0012-7094,1547-7398},
   MRCLASS = {99-04},
  MRNUMBER = {1546121},
       DOI = {10.1215/S0012-7094-39-00522-3},
       URL = {https://doi.org/10.1215/S0012-7094-39-00522-3},
}

@article {shalom,
    AUTHOR = {Shalom, Yehuda},
     TITLE = {Rigidity of commensurators and irreducible lattices},
   JOURNAL = {Invent. Math.},
  FJOURNAL = {Inventiones Mathematicae},
    VOLUME = {141},
      YEAR = {2000},
    NUMBER = {1},
     PAGES = {1--54},
      ISSN = {0020-9910,1432-1297},
   MRCLASS = {22E40 (20E08 20F65 20G10)},
  MRNUMBER = {1767270},
MRREVIEWER = {Alain\ Valette},
       DOI = {10.1007/s002220000064},
       URL = {https://doi.org/10.1007/s002220000064},
}

@article {soma:free,
    AUTHOR = {Soma, Teruhiko},
     TITLE = {Existence of non-{B}anach bounded cohomology},
   JOURNAL = {Topology},
  FJOURNAL = {Topology. An International Journal of Mathematics},
    VOLUME = {37},
      YEAR = {1998},
    NUMBER = {1},
     PAGES = {179--193},
      ISSN = {0040-9383},
   MRCLASS = {55N35 (53C20)},
  MRNUMBER = {1480885},
MRREVIEWER = {Yuli\ B.\ Rudyak},
       DOI = {10.1016/S0040-9383(97)00002-5},
       URL = {https://doi.org/10.1016/S0040-9383(97)00002-5},
}

@article {soma:srf,
    AUTHOR = {Soma, Teruhiko},
     TITLE = {The zero-norm subspace of bounded cohomology},
   JOURNAL = {Comment. Math. Helv.},
  FJOURNAL = {Commentarii Mathematici Helvetici},
    VOLUME = {72},
      YEAR = {1997},
    NUMBER = {4},
     PAGES = {582--592},
      ISSN = {0010-2571,1420-8946},
   MRCLASS = {57N65 (55N35)},
  MRNUMBER = {1600150},
MRREVIEWER = {Darryl\ McCullough},
       DOI = {10.1007/s000140050035},
       URL = {https://doi.org/10.1007/s000140050035},
}

\end{document}